\documentclass{article}
\usepackage[utf8]{inputenc}
\usepackage{fullpage}
\usepackage{amssymb}
\usepackage{amsmath}
\usepackage{amsthm}
\usepackage{enumerate}
\usepackage{CJK}
\usepackage{color}
\usepackage{mathrsfs}
\usepackage{array}
\newtheorem{theorem}{Theorem}[section]
\newtheorem{proposition}[theorem]{Proposition}
\newtheorem{definition}[theorem]{Definition}
\newtheorem{corollary}[theorem]{Corollary}
\newtheorem{lemma}[theorem]{Lemma}

\numberwithin{equation}{section} \numberwithin{theorem}{section}
\numberwithin{equation}{section}

\newtheorem{example}[theorem]{Example}
\numberwithin{equation}{section}

\newcommand{\vertrule}[1][1ex]{\rule{.6pt}{#1}}
\newcommand{\vbl}{\hskip0.05cm \vertrule[1.5ex]\hskip0.05cm}
\newcommand{\vbs}{\hskip0.05cm \vertrule[1.3ex]\hskip0.05cm}
\usepackage{cancel}
\usepackage{ulem}
\usepackage{soul}

\DeclareMathOperator*\inte{int}

\DeclareMathOperator*\co{co}

\DeclareMathOperator*\gph{gph}

\usepackage{cancel}
\usepackage{ulem}
\usepackage{soul}
\usepackage{comment}

\title{Relative Well-Posedness of Truncated Constrained Systems Accompanied\\ by Variational Calculus}

\author{Boris S. Mordukhovich\thanks{Department of Mathematics, Wayne State University
Detroit MI 48202 (e-mail: aa1086@wayne.edu). } 
\and Pengcheng Wu\thanks{Department of Applied Mathematics, The Hong Kong Polytechnic University, Hong Kong (e-mail: pcwu0725@163.com).}
\and Xiaoqi Yang\thanks{Department of Applied Mathematics, The Hong Kong Polytechnic University, Kowloon, Hong Kong (mayangxq@polyu.edu.hk).}}

\begin{document}
\maketitle

{\bf Abstract:} 
The paper concerns foundations of sensitivity and stability analysis in optimization and related areas, being primarily addressed truncated constrained systems. We consider general models, which are described by multifunctions between Banach spaces and concentrate on characterizing their well-posedness properties that revolve around Lipschitz stability and metric regularity relative to sets. Invoking tools of variational analysis and generalized differentiation, we introduce new robust notions of relative contingent coderivatives. The novel machinery of variational analysis leads us to establishing complete characterizations of such properties and developing basic rules of variational calculus interrelated with the obtained characterizations of well-posedness. Most of the our results valid in general infinite-dimensional settings are also new in finite dimensions.


{\bf AMS:} 49J53, 49J52, 49K40

\section{Introduction}\label{intro} This paper is devoted to the study of stability and related well-posedness properties of general systems described by truncated constrained set-valued mappings/multifunctions between Banach spaces. The main difference of the novel {\it relative well-posedness} notions for multifunctions studied here and their known prototypes is that now we naturally and effectively incorporate constraints and truncations, which are crucial in applications and are highly challenging mathematically.

Although most of the results obtained below are new even in 
finite dimensions, we address here general infinite-dimensional settings. The main motivation for this (besides mathematical reasons) is that {\it infinite-dimensional models} frequently and naturally appear in various areas of constrained optimization and related topics. Among them, we mention systems control, semi-infinite programming, dynamic optimization, finance, stochastic programming, transportation networks, etc. Sensitivity analysis of such truncated constrained models and their optimal solutions is of the highest priority for practical applications. Saying this, it should be emphasized that---as seen below---the infinite-dimensional analysis of relative well-posedness is much more involved in comparison with its finite-dimensional counterpart and requires a fairly delicate machinery. Providing the formulations and discussions of the relative well-posedness properties, we employ advanced tools of {\it variational analysis} and {\it generalized differentiation} for their study and characterizations. To proceed in this direction, new notions of {\it relative} coderivatives for multifunctions are introduced and investigated in the paper. These notions are certainly  of their own interest for variational theory and applications, while the main goal here  is to implement them deriving {\it complete characterizations} of the relative well-posedness properties under consideration.

Well-posedness properties of (mainly unconstrained) multifunctions and their coderivative characterizations have been broadly developed together with numerous applications to various problems in nonsmooth optimization, variational inequalities, and optimal control in finite and infinite dimensions. We refer the reader to the books \cite{Borwein2005,ioffe,Mordukhovich2006,m18,m24,Rockafellar1998,thibault} 
with the vast bibliographies and commentaries therein. { For constrained systems, the notions of metric regularity, directional metric (sub)regularity and isolated calmness were investigated in, e.g.,  \cite{arutyunov2006,ioffe} by using strong slopes and in \cite{Benko2020,gfrerer2013} by using directional limiting coderivatives, respectively. In particular, the relative isolated calmness (namely, a recovery bound property) of various regularization models has been extensively investigated in machine learning and compressed sensing; see, e.g., \cite{Bickel09,Hu2017,Meinshausen2009La} and the references therein. 
It was also noted in \cite{gfrerer2013} that for the fulfillment of certainty stationarity conditions, we only need a regular behavior of the constrained systems with respect to one single critical direction but not on the whole space. The Lipschitz-like property relative to a set was studied in \cite{Benko2020} and \cite{Yang2021,Yao2023,Yao2023-2} by using directional limiting coderivatives and projectional coderivatives, respectively. Specifically, 
it has been shown that constrained coderivatives can provide sufficient conditions for stability characterizations of solution mappings of affine variational inequalities \cite{Yao2023} and the level sets for extended real-valued functions \cite{Yang2021}.  We will compare our new results with those known in the literature. Moreover, after the initial submission of this paper, a number of follow-up works have been developed; see, in particular, \cite{qin2024,thinh2024b,thinh2024a}. All of this shows a growing interest of the optimization community to the topics under consideration.}\vspace*{0.01in}

This paper introduces novel {\it relative contingent coderivatives} of set-valued mappings between Banach spaces and { establishes} {\it complete coderivative characterizations} of the relative Lipschitz-like property, as well as the { relative metric regularity} and relative linear openness of a set-valued mapping, under a newly defined relative partial sequential normal compactness (PSNC) condition. The efficient usage of the established stability criteria for structured multifunctions appearing in truncated constrained problems and applications requires {\it variational calculus}. We derive comprehensive {\it pointbased chain} and {\it sum rules} for {\it relative contingent coderivatives}, the limiting constructions from the obtained characterizations of the relative Lipschitz-like property for general multifunctions, as well as that for the relative PSNC. Our proofs are mainly based on the newly established {\it relative} versions of {\it extremal principle} and {\it fuzzy intersection rule}. It is important to emphasize that the obtained {\it characterizations of well-posedness} play a significant role in verifying the {\it fulfillment of new calculus rules} for major classes of truncated constrained multifunctions.\vspace*{0.02in}

The {\it main contributions} of the paper are as follows:\vspace*{-0.01in}

$\bullet$ We introduce novel {\it relative contingent coderivatives} of set-valued mappings by using the intersection of $\varepsilon$-normal set to the graph of the truncated constrained multifunction with the images of the tangent cones of the constrained/truncated sets under the respective duality mappings and then establish in these terms {\it complete coderivative characterizations} of the relative Lipschitz-like property of multifunctions between reflexive Banach spaces.

$\bullet$ We prove the {\it equivalences} (up to taking the inverse) between the {\it relative Lipschitz-like property}, { \it relative metric regularity}, and {\it relative linear openness}. As an interesting application of these results, we reveal a new relationship between weakly and strongly convergent sequences.

$\bullet$ We establish novel {\it variational calculus rules} for {\it relative contingent coderivatives} and {\it relative PSNC property}. These rules, together with the characterizations of well-posedness, are instrumental to facilitate applications of relative contingent coderivatives to broad classes of truncated constrained systems in variational analysis.\vspace*{0.01in}

The rest of the paper is organized as follows. In Section~\ref{sec:gen-diff}, we first recall needed  preliminaries from variational analysis  and then introduce new  relative coderivative constructions with deriving some of their important properties used to establish major results below. Section~\ref{sec:well-posed} presents central results of the paper. We introduce there the {\it relative Lipschitz-like, metric regularity}, and {\it linear openness} properties of truncated multifunctions and obtain their {\it complete coderivative characterizations} in terms of the new constructions from Section~\ref{sec:gen-diff}.
In Section \ref{sec:calculus}, we establish {\it variational calculus rules} for relative contingent coderivatives and  relative PSNC property. \vspace*{-0.05in}

\section{New Tools of Generalized Differentiation}\label{sec:gen-diff}
We start this section with recalling some preliminaries from variational analysis and generalized differentiation in Banach spaces by following mainly the books \cite{Borwein2005,Mordukhovich2006}. Throughout the paper, the {\it standard notation} of variational analysis and generalized differentiation is used; cf.\ \cite{Borwein2005,Mordukhovich2006,Rockafellar1998}. Recall that the norms of a Banach space $X$ and its topological dual $X^*$ are denoted by $\|\cdot\|_X$ and $\|\cdot\|_{X^*}$, with the closed unit balls $B_X\subset X$ and $B_{X^*}\subset X^*$, respectively. The notation $\mathbb S_X=\{x\in X~|~\|x\|_X=1\}$ signifies the unit sphere in $X$. For $x\in X$ and $x^*\in X^*$, the symbol $\langle x^*,x\rangle$ stands for the canonical duality between $X^*$ and $X$. We use the notation $``x_k\to x"$ to indicate the strong convergence of a sequence $x_k$ in $X$ to $x$, $``x_k\rightharpoonup x"$ for the weak convergence of a sequence $x_k$ in $X$ to $x$ (i.e., $\langle x^*,x_k\rangle\to\langle x^*,x\rangle$ for any $x^*\in X$), and $``x_k^*\rightharpoonup^* x^*"$ to denote the weak$^*$ convergence of a sequence $x_k^*$ in $X^*$ to $x^*$ (i.e., $\langle x_k^*,x\rangle\to\langle x^*,x\rangle$ for any $x\in X$). Given a subset $\Omega \subset X$, the affine hull of $\Omega$ is denoted by aff($\Omega$). The symbol $``x_k\stackrel{\Omega}{\longrightarrow}x"$ is used to indicate that $x_k\to x$ with $x_k\in\Omega$ for $k \in \mathbb N:=\{1,2,\ldots\}$. The notation $\co\Omega$ signifies the convex closure of $\Omega$. Given $\bar x\in X$, we use ${\cal N}(\bar x)$ to signify the collection of all (open) neighborhoods of $\bar x$. Finally, denote $a^+:=\max\{a,0\}$ for $a\in\mathbb R$.

For Banach spaces $X$ and $Y$, $\mathcal{L}(X,Y)$ denotes the space of all bounded linear operators from $X$ to $Y$. Given $A \in\mathcal{L}(X,Y)$,  the norm of $A$ is defined by $\|A\|_{\mathcal{L}(X,Y)}: = \sup_{x \in X \setminus \{0\}} \|A(x)\|_Y/\|x\|_X.$ We say that $\Omega\subset X$ is {\it closed around} a point $\bar x\in\Omega$ if the intersection $\Omega\cap U$ is closed for some closed set $U$ containing $\bar x$ as an interior point. The interior of $\Omega$ is denoted by $\operatorname{int} \Omega$. The {\it indicator function} $\iota_\Omega$ of $\Omega$ is defined by $\iota_\Omega(x):=0$ if $x\in\Omega$ and $\infty$ otherwise.
The {\it distance function} associated with $\Omega$ is given by $d(x,\Omega):=\inf\limits_{z\in\Omega}\|z-x\|_X, \ x \in X,$ where $d(x,\Omega):=\infty$ if $\Omega= \emptyset$  by convention.
For a set $K\subset X$, we define $\mbox{proj}_\Omega(K):=\big\{z\in\Omega~\big|~\exists x\in K~\mbox{with}~\|z-x\|_X=d(x,\Omega)\big\},$
which reduces to the standard point projection $\mbox{proj}_\Omega(x)$ if $K=\{x\}$.\vspace*{-0.05in}

\begin{definition}
Let $S:X\rightrightarrows Y$ be a multifunction between Banach spaces. Then for any point $\bar x$ from its domain ${\rm dom}\,S:=\{x\in X\;|\;S(x)\ne\emptyset\}$, we have:

{\bf(i)} The {\sc strong} and {\sc weak}$^*$ $($if $Y=Z^*$ for some Banach space Z$)$ {\sc outer limits} of $S$ are defined, respectively, by
$$
\begin{aligned}
&s\mbox{-}\limsup\limits_{x\to\bar x}S(x):=\big\{y\in Y~\big|~\exists x_k\to\bar x~\operatorname{and}~y_k\in S(x_k)~\operatorname{such~that}~y_k\to y\big\},\\ &w^*\mbox{-}\limsup\limits_{x\to\bar x}S(x):=\big\{y\in Y~\big|~\exists x_k\to\bar x~\operatorname{and}~y_k\in S(x_k)~\operatorname{such~that}~y_k\rightharpoonup^* y\big\}.
\end{aligned}
$$\vspace*{-0.15in}

 {\bf(ii)} The {\sc strong inner limit} of $S$ is given by
$$
s\mbox{-}\liminf\limits_{x\to\bar x}S(x):=\big\{y\in Y~\big|~\forall x_k\to\bar x,~\exists y_k\in S(x_k)~\operatorname{such~that}~y_k\to y\big\}.
$$
\end{definition}\vspace*{-0.05in}

Recall now some notions of tangent and normal approximations.\vspace*{-0.05in}

\begin{definition}\label{cone} Let $\Omega$ be a subset of a Banach space $X$ with $\bar x\in\Omega$.

 {\bf(i)} The $($Bouligand-Severi$)$ {\sc tangent/contingent cone} to $\Omega$ at $\bar x$ is defined by 
 $T(\bar x;\Omega):=s\mbox{-}\limsup\limits_{t\downarrow0}\ (\Omega - \bar x)/t.$
 
{\bf(ii)} Given $\varepsilon \geq 0$, the set of {\sc $\varepsilon$-normals} to $\Omega$ at $\bar x$ is
\begin{equation*}
\widehat N_\varepsilon(\bar x;\Omega):=\Big\{x^*\in X^*~\Big|~\limsup\limits_{x\stackrel{\Omega}{\longrightarrow}\bar x}\frac{\langle x^*,x-\bar x\rangle}{\|x-\bar x\|_X}\leq\varepsilon\Big\}.
\end{equation*}
For $\varepsilon=0$, the set $\widehat N(\bar x;\Omega):=\widehat N_0(\bar x;\Omega)$ is a cone, which is called the $($Fr\'echet$)$ {\sc regular normal cone} to $\Omega$ at $\bar x$. { It is obvious that $\widehat N_\varepsilon(\bar x;\Omega)$ is positively homogeneous with respect to $\varepsilon$, i.e., $\lambda\widehat N_\varepsilon(\bar x;\Omega)=\widehat N_{\lambda\varepsilon}(\bar x;\Omega)$ for all $\lambda>0$.}

{\bf(iii)} The $($Mordukhovich$)$ {\sc limiting normal cone} to $\Omega$ at $\bar x$ is
\begin{equation}\label{lnc}
N(\bar x;\Omega):=w^*\mbox{-}\limsup\limits_{x\stackrel{\Omega}{\longrightarrow}\bar x,\,\varepsilon\downarrow0}\widehat N_\varepsilon(x;\Omega).
\end{equation}

{\bf(iv)} The set $\Omega$ is said to be $($normally$)$ {\sc regular} at $\bar x$ if $N(\bar x;\Omega) = \widehat N(\bar x;\Omega)$.
\end{definition}

For a mapping $S:X\rightrightarrows Y$, consider its graph
$\operatorname{gph} S:=\{(x,y)~|~y\in S(x)\}$ and {\it kernel}  $
\ker S:=\big\{x\in X~\big|~0\in S(x)\big\}$. We say that $S$ is {\it positively homogeneous} if
\begin{equation}\label{homo}
0\in S(0)~~\operatorname{and}~~S(\lambda x)=\lambda S(x)~~\operatorname{for~all}~\lambda>0~\operatorname{and}~x \in X
\end{equation}
meaning that $\operatorname{gph} S$ is a cone. If $S$ is a positively homogeneous multifunction, the {\it outer norm} of $S$ is defined by $|S|^+:=\sup\limits_{x\in B_X}\sup\limits_{y\in S(x)}\|y\|_Y
$. Given nonempty sets $\Omega\subset X$ and $\Theta\subset Y$, the {\it truncated restriction} of $S$ on $\Omega$ by $\Theta$ is
\begin{equation*}
S|_\Omega^\Theta: x \mapsto \left \{\begin{aligned}
&S(x)\cap\Theta~~&&\operatorname{if}~x\in \Omega,\\
&\emptyset~~&&\operatorname{if}~x\notin \Omega.
\end{aligned}
\right.
\end{equation*}
In the case of parametric optimization, the truncation set $\Theta$ can be viewed as a set of optimal solutions to a lower-level problem in the setting of {\it bilevel optimization}: $\min \iota_\Theta(y)\;\mbox{ subject to }\;y \in S|_\Omega(x)$. This means  that  $y \in S|_\Omega^\Theta(x)$ if and only if $y$ is a solution to the bilevel optimization problem with $\iota_\Theta(y)=0$. Observe that
$$
\operatorname{gph} S|_\Omega^\Theta=\operatorname{gph} S\cap[\Omega\times \Theta],\quad~~\operatorname{dom} S|_\Omega^\Theta=\Omega\cap S^{-1}(\Theta).
$$
{ and then use the notation $S|_\Omega:=S|_\Omega^Y$ and $S|^\Theta:=S|_X^\Theta$ in what follows.}

Before defining our new coderivative constructions for multifunctions, we recall the notion of the {\it duality mapping} $J_X:X\rightrightarrows X^*$ between a Banach space $X$ and its topological dual space $X^*$ defined by
\begin{equation}\label{dual map}
J_X(x):=\big\{x^*\in X^*~\big|~\langle x^*,x\rangle=\|x\|_X^2,~\|x^*\|_{X^*}=\|x\|_{X}\big\}\;\mbox{ for all }\;x\in X.
\end{equation}
Observe from the definition that $J_X(\lambda x)=\lambda J_X(x)$ for all $\lambda\in\mathbb R$ and $x\in X$. It follows from \cite[Proposition~2.12]{Bonnans2000} that $J_X(x)\not=\emptyset$ whenever $x\in X$. By the results due  to Trojanski \cite{Troyanski1971}, a reflexive Banach space $X$ (this means $X\cong X^{**}$) can be renormed so that $X$ and $X^*$ are both locally uniformly convex. Throughout this paper, we say that $X$ is a {\it reflexive Banach space} meaning that $X$ is a reflexive Banach space which is renormed in such a way that both $X$ and $X^*$ are locally uniformly convex. Furthermore, the result of { \cite[Proposition~3.6]{Cioranescu1990} tells us that $J_X:X\to X^*$ is {\it homeomorphic} (bijective and bicontinuous)} when $X$ is a reflexive Banach space. Note that $J_X=id_X$ (identity mapping of $X$) if $X$ is a Hilbert space. The duality mapping is instrumental for the underlying coderivative constructions defined below providing a {\it bridge} between the contingent cone $T(\bar x;\Omega)$ in the original space $X$ and the set of $\varepsilon$-normals $\widehat N_\varepsilon(\bar x;\Omega)$ in the dual space $X^*$.

{  Next we present two results concerning the duality mapping $J_X$, which will be used in the sequel. In particular, the second one shows that, in a reflexive Banach space, the tangent cones of closed convex sets preserve the inner semicontinuity property under the action of the duality mapping.\vspace*{-0.02in}

\begin{lemma}\label{JTN}
Let $\Omega\subset X$, $x\in\Omega$, and $\varepsilon\geq0$. Suppose that $x^*\in\widehat N_\varepsilon(x;\Omega)\cap J_X(T(x;\Omega))$. Then we have $\|x^*\|_{X^*}\leq\varepsilon$.
\end{lemma}\vspace*{-0.05in}
\begin{proof}
Pick $x^*\in \widehat N_\epsilon(x;\Omega) \cap J_X\big(T(x;\Omega)\big)$ and get from $x^*\in J_X\big(T(x;\Omega)\big)$ that there exists $z\in T(x;\Omega)$ such that $\langle x^*,z\rangle=\|x^*\|_{X^*}^2$ and $\|z\|_X=\|x^*\|_{X^*}$. This allows us to find sequences $t_k\downarrow0$ and $x_k\stackrel{\Omega}{\longrightarrow}\bar x$ with $(x_k-\bar x)/t_k\to z$ as $k\to\infty$. The case where $z=0$ is trivial. If $z\not=0$, then we have $(x_k-\bar x)/\|x_k-\bar x\|_X\to z/\|z\|_X$ as $k\to\infty$, which being combined with $x^*\in \widehat N_\epsilon(x;\Omega)$ tells us that
$$
\|x^*\|_{X^*}=\left\langle x^*,\frac{z}{\|z\|_{X}}\right\rangle=\lim\limits_{k\to\infty}\frac{\langle x^*,x_k-\bar x\rangle}{\|x_k-\bar x\|_X}\leq\varepsilon
$$
and thus completes the proof of the lemma.
\end{proof}}\vspace*{-0.05in}

\begin{lemma}\label{JT}
Let $X$ be a reflexive Banach space, $\Omega\subset X$ be a closed and convex set, and $x^*\in J_X(T(\bar x;\Omega))$. Then for any $\varepsilon>0$, there exists $\delta>0$ with
$$
x^*\in J_X\big(T(x;\Omega)\big)+\varepsilon B_{X^*}\;\mbox{ whenever }\;x\in\Omega\cap(\bar x+\delta B_X).
$$
\end{lemma}
\begin{proof}
Suppose on the contrary that the claimed inclusion fails and then find a number $\varepsilon_0>0$ and a sequence $x_k\stackrel{\Omega}{\longrightarrow}\bar x$ as $k\to\infty$ such that
\begin{equation}\label{notin}
x^*\notin J_X\big(T(x_k;\Omega)\big)+\varepsilon_0 B_{X^*}\;\mbox{ for all }\;k\in\mathbb N.
\end{equation}
By the aforementioned properties of the duality mapping, we deduce from $J_X^{-1}(x^*)\in T(\bar x;\Omega)$ and the well-known inner semicontinuity of the tangent cone (see, e.g., \cite[page~47]{Bonnans2000})
$
T(\bar x;\Omega)\subset s\mbox{-}\liminf\limits_{x\stackrel{\Omega}{\longrightarrow}\bar x}T(x;\Omega),
$
valid for convex sets $\Omega$, that there exists a sequence of $\{z_k\}$ with $z_k\in T(x_k;\Omega)$ such that $z_k\to J_X^{-1}(x^*)$ as $k\to \infty$. Letting $x_k^*:=J_X(z_k)$, we get that
$
x_k^*\in J_X\big(T(x_k;\Omega)\big)~~\operatorname{and}~~x_k^*\to x^*\;\mbox{ as }\;k\to\infty,
$
which contradicts \eqref{notin} and thus completes the proof.
\end{proof}

{ 
The duality mapping and (sub)differentiation of the norm in a Banach space are closely related. Consider the product space $X\times Y$ with the norm $\|(x,y)\|_{X\times Y}:=\sqrt{\|x\|_X^2+\|y\|_Y^2}$. It is easy to check that the duality mapping of $X\times Y$ satisfies $J_{X\times Y}(x,y)=(J_X(x), J_Y(y))$. Combining this with \cite[Example~3.36]{Nam}, gives us
$$
\partial\|(\cdot,\cdot)\|_{X\times Y}(x,y)=\left\{
\begin{aligned}
&B_{X^*\times Y^*}~~&&\mbox{if}~ (x,y)=(0,0),\\
&\left(\frac{J_X(x)}{\|(x,y)\|_{X\times Y}},\frac{J_Y(y)}{\|(x,y)\|_{X\times Y}}\right)~~&&\mbox{if}~ (x,y)\not=(0,0).
\end{aligned}
\right.
$$
In the case of reflexive Banach spaces $X$ and $Y$, it is well known that $\|(x,y)\|_{X\times Y}$ is Fr$\acute{\text{e}}$chet differentiable at any nonzero point. When $(x,y)=(0,0)$, we obviously have that $\|(x,y)\|^2_{X\times Y}$ is Fr$\acute{\text{e}}$chet differentiable, and thus the function $\|(x,y)\|^2_{X\times Y}$ is Fr$\acute{\text{e}}$chet differentiable on the whole space $X\times Y $ with the derivative
\begin{equation}\label{diffduality}
\nabla \|(\cdot,\cdot)\|^2_{X\times Y}(x,y)=2\left(J_X(x),J_Y(y)\right).
\end{equation}}\vspace*{-0.1in}

Now we are ready to define new relative contingent coderivatives for multifunctions that are used below. In comparing with the recent {\it projectional coderivative} studied in \cite{Yang2021,Yao2023}, where the projection of a normal cone to the tangent cone of the constrained set was used, we are motivated by the fact that the intersection of the normal cone with the tangent cone would be easier to deal with than to find projections. Observe that the novel coderivatives below are the first ones in variational analysis that combine {\it primal/tangent} and {\it dual/normal} constructions in infinite dimensions. It is worth noting that apart from the multifunction $S$, another two components in the definition below are the sets $\Omega$ and $\Theta$. The set $\Omega$ represents the restriction on the parameters, while the set $\Theta$ can be understood as a truncation of the image of $S$.  \vspace*{-0.03in}

\begin{definition}\label{coderivatives}
Let $S:X\rightrightarrows Y$ be a multifunction between Banach spaces, $\Omega\subset X$, $\Theta\subset Y$ be nonempty, and $(\bar x,\bar y)\in \operatorname{gph} S|_\Omega^\Theta$.

{\bf(i)} Given $\varepsilon\geq0$, the {\sc $\varepsilon$-contingent coderivative} of $S$ {\sc relative to} $\Omega\times\Theta$ at $(\bar x,\bar y)$ is defined as a multifunction $\widehat D_\varepsilon^*S_\Omega^\Theta(\bar x|\bar y):Y^*\rightrightarrows X^*$ with the values
\begin{equation}\label{harnepsilion'}
\begin{array}{ll}
\widehat D_\varepsilon^*S_\Omega^\Theta(\bar x\vbl\bar y)(y^*):=\big\{x^*\in X^*~\big|&(x^*,-y^*)\in\widehat N_\varepsilon\big((\bar x,\bar y);{ \operatorname{gph} S|_\Omega^\Theta}\big)\\
&\cap\big[J_X\big(T(\bar x;\Omega)\big)\times J_Y\big(T(\bar y;\Theta)\big)\big]\big\}.
\end{array}
\end{equation}

{\bf(ii)} The {\sc normal contingent coderivative} of $S$ {\sc relative to} $\Omega\times\Theta$ at $(\bar x,\bar y)$ is a multifunction $D_N^*S_\Omega^\Theta(\bar x\vbl\bar y):Y^*\rightrightarrows X^*$ defined by
\begin{equation}\label{normalcoder}
D_N^*S_\Omega^\Theta(\bar x\vbl \bar y)(\bar y^*):=w^*\mbox{-}\limsup\limits_{\stackrel{\stackrel{(x,y)\stackrel{{ \operatorname{gph} S|_\Omega^\Theta}}{\longrightarrow}(\bar x,\bar y)}{y^*\rightharpoonup^* \bar y^*}}{\varepsilon\downarrow0}}
\widehat D_\varepsilon^*S_\Omega^\Theta(x\vbs y)(y^*).
\end{equation}
That is, \eqref{normalcoder} is the collection of such $\bar x^*\in X^*$ for which there exist sequences $\varepsilon_k\downarrow 0$, $(x_k,y_k)\stackrel{{ \operatorname{gph} S|_\Omega^\Theta}}{\longrightarrow}(\bar x,\bar y)$, and $(x_k^*,y_k^*)\rightharpoonup^*(\bar x^*,\bar y^*)$ with
\begin{equation}\label{cod-seq}
(x_k^*,-y_k^*)\in\widehat N_{\varepsilon_k}\big((x_k,y_k);{ \operatorname{gph} S|_\Omega^\Theta}\big)\cap\big[J_X\big(T(x_k;\Omega)\big)\times J_Y\big(T(y_k;\Theta)\big)\big],\quad k\in\mathbb N.
\end{equation}

{\bf(iii)} The {\sc mixed contingent coderivative} of $S$ {\sc relative to} $\Omega\times\Theta$ at $(\bar x,\bar y)$ is a multifunction $D_M^*S_\Omega^\Theta(\bar x\vbl\bar y):Y^*\rightrightarrows X^*$ defined by
\begin{equation}\label{mixedcoder}
D_M^*S_\Omega^\Theta(\bar x\vbl\bar y)(\bar y^*):=w^*\mbox{-}\limsup\limits_{\stackrel{\stackrel{(x,y)\stackrel{{ \operatorname{gph} S|_\Omega^\Theta}}{\longrightarrow}(\bar x,\bar y)}{y^*\rightarrow\bar y^*}}{\varepsilon\downarrow 0}}
{ \widehat D_\varepsilon^*S_\Omega^\Theta(x\vbs y)(y^*)}.
\end{equation}
That is, \eqref{mixedcoder} is the collection of such $\bar x^*\in X^*$ for which there exist sequences $\varepsilon_k\downarrow0$, $(x_k,y_k)\stackrel{{ \operatorname{gph} S|_\Omega^\Theta}}{\longrightarrow}(\bar x,\bar y)$, $y_k^*\to\bar y^*$, and $x_k^*\rightharpoonup^*\bar x^*$ with $(x^*_k,y^*_k)$ satisfying \eqref{cod-seq}.

{\bf{(iv)}} The {\sc mirror contingent coderivative} of $S$ {\sc relative to} $\Omega\times\Theta$ at $(\bar x,\bar y)$ is a multifunction $D_m^*S_\Omega^\Theta(\bar x\vbl \bar y):Y^*\rightrightarrows X^*$ defined by
\begin{equation}\label{mirrorcoder}
D_m^*S_\Omega^\Theta(\bar x\vbl\bar y)(\bar y^*):=s\mbox{-}\limsup\limits_{\stackrel{\stackrel{(x,y)\stackrel{{ \operatorname{gph} S|_\Omega^\Theta}}{\longrightarrow}(\bar x,\bar y)}{y^*\rightharpoonup^* y^*}}{\varepsilon\downarrow 0}}
\widehat D_\varepsilon^*S_\Omega^\Theta(x\vbs y)(y^*).
\end{equation}
That is, \eqref{mirrorcoder} is the collection of such $\bar x^*\in X^*$ for which there exist sequences $\varepsilon_k\downarrow0$, $(x_k,y_k)\stackrel{{ \operatorname{gph} S|_\Omega^\Theta}}{\longrightarrow}(\bar x,\bar y)$, $y_k^*\rightharpoonup^*\bar y^*$, and $x_k^*\to\bar x^*$ with $(x^*_k,y^*_k)$ satisfying \eqref{cod-seq}.
\end{definition}

In the above notions, $\bar y$ is dropped if $S(\bar x)=\{\bar y\}$. For $\widehat D_\varepsilon^*S_\Omega^\Theta(\bar x\vbl\bar y)(y^*)$, if $\Omega=X$, we drop $\Omega$ and write it $D_\varepsilon^*S^\Theta(\bar x\vbl\bar y)$ and if $\Theta = Y$, we write as $D_\varepsilon^*S_\Omega(\bar x\vbl\bar y)$. Such simplifications also apply when the subscript $\varepsilon$ is replaced by $N$, $M$ and $m$. Without $\Omega$ and $\Theta$, the constructions of Definition~\ref{coderivatives} reduce to those in \cite{Mordukhovich2006}. 

We obviously have the inclusion
\begin{equation}\label{in}
\begin{aligned}
\widehat D_\varepsilon^*S_\Omega^\Theta(\bar x\vbl\bar y)(y^*)\subset \widehat D_\varepsilon^*S|_\Omega^\Theta(\bar x|\bar y)(y^*)\mbox{ for all }\;y^*\in Y^*,
\end{aligned}
\end{equation}
where $\widehat D_\varepsilon^*S|_\Omega^\Theta(\bar x\vbl\bar y)(y^*)$ (with a stroke between $S$ and $^\Theta_\Omega$) is for the $\epsilon$-mixed coderivative of multifunction $S|_\Omega^\Theta$ in the sense of \cite{Mordukhovich2006}, in which no intersection with the images of the tangent cones under the duality mapping was employed. The similar inclusions are valid when the subscript $\varepsilon$ is replaced by $N$, $M$ and $m$. We also have
\begin{equation}\label{cod-normality}
\widehat D_0^*S_\Omega^\Theta(\bar x\vbl\bar y)(y^*)\subset D_M^*S_\Omega^\Theta(\bar x\vbl\bar y)(y^*)\subset D_N^*S_\Omega^\Theta(\bar x\vbl\bar y)(y^*)\;\mbox{ for all }\;y^*\in Y^*,
\end{equation}
where the second inclusion holds as an equality provided that ${\rm dim}\,Y<\infty$. Replacing $D_M^*$ by $D_m^*$, both inclusions in \eqref{cod-normality} are satisfied with the second one holding as an equality when ${\rm dim}\,X<\infty$. The property
\begin{equation}\label{cod-nor}
|D_M^*S_\Omega^\Theta(\bar x\vbl\bar y)|^+=|D_N^*S_\Omega^\Theta(\bar x\vbl\bar y)|^+
\end{equation}
is called the {\it coderivative normality} of $S$ {\it relative to} $\Omega \times \Theta$ at $(\bar x,\bar y)\in\operatorname{gph} S|_\Omega^\Theta$. Some conditions ensuring the equalities in \eqref{cod-normality}, and hence the relative coderivative normality \eqref{cod-nor}, can be derived similarly to the unconstrained case as in \cite[Proposition~4.9]{Mordukhovich2006}.

The relative contingent coderivatives \eqref{normalcoder} and \eqref{mixedcoder} are generalized differential constructions, which allow us to provide {\it complete characterizations} of the relative well-posedness notions for constrained multifunctions in Section~\ref{sec:well-posed}. Both of these coderivative multifunctions are {\it positively homogeneous} and enjoy comprehensive {\it calculus rules}, which are developed in Section~\ref{sec:calculus}.

Recall that a mapping $f\colon X\to Y$ is {\it strictly differentiable} at $\bar x$ if
\begin{equation}\label{strict-diff}
\lim\limits_{(x,z)\to (\bar x, \bar x)}\frac{f(x)-f(z)-\nabla f(\bar x)(x-z)}{\|x-z\|_X}=0,
\end{equation}
where $\nabla f(\bar x): X \to Y$ is a bounded linear operator called the {\it strict derivative} of $f$ at $\bar x$. It is clear that \eqref{strict-diff} holds if $f$ is ${\cal C}^1$-smooth around $\bar x$ but not vice versa. 

Given an extended-real-valued function $\varphi\colon X\to\overline{\mathbb R}:=(-\infty,\infty]$ in a Banach space $X$ with $\varphi(\bar x)<\infty$, the (Fr\'echet) {\it regular subdifferential} of $\varphi$ at $\bar x$ is defined by
\begin{equation}\label{reg-sub}
\widehat\partial\varphi(\bar x):=\Big\{x^*\in X^*~\Big|~{ \liminf\limits_{x\to\bar x}}\frac{\varphi(x)-\varphi(\bar x)-\langle x^*,x-\bar x\rangle}{\|x-\bar x\|_X}\ge 0\Big\}.
\end{equation}
This notion is closely related to the regular normal cone $\widehat N(\bar x;\Omega)$ from Definition~\ref{cone} with $\varepsilon=0$. In particular, the regular subdifferential of the indicator function $\iota_\Omega$ at $\bar x\in\Omega$ is $\widehat N(\bar x;\Omega)$. If $\varphi$ is Fr\'echet differentiable at $\bar x$, then \eqref{reg-sub} is a singleton, which is the classical Fr\'echet derivative of $\varphi$ at $\bar x$, while the regular subdifferential of a convex function reduces to the subdifferential of convex analysis.\vspace*{0.03in}

The following proposition follows \cite[Lemma~3.1]{Mordukhovich2006} by replacing $\widehat N(\bar x;\Omega_1\cap\Omega_2)$ there with $\widehat N_\varepsilon(\bar x;\Omega_1\cap\Omega_2)$.

\begin{proposition}\label{fuzzy}
Let $X$ be a reflexive Banach space and $\Omega_1,\;\Omega_2\subset X$ be  locally closed around $\bar x\in\Omega_1\cap\Omega_2$. Suppose that $x^*\in\widehat N_\varepsilon(\bar x;\Omega_1\cap\Omega_2)$ with some $\varepsilon>0$. Then for any $\gamma>0$ there exist $\lambda\geq0$, $x_i\in\Omega_i\cap(\bar x+\gamma B_X)$, and $x_i^*\in \widehat N(x_i;\Omega_i)+(\varepsilon+\gamma) B_{X^*}$, $i=1,2$, such that $\lambda x^*=x^*_1+x^*_2$ and $\max\{\lambda,\|x^*_1\|\}=1$.
\end{proposition}\vspace*{-0.05in}

\section{Relative Well-Posedness Properties and Coderivative Characterizations}\label{sec:well-posed}\vspace*{-0.05in}

In this section, we define the underlying {\it well-posedness} properties of set-valued mappings between Banach spaces {\it relative} to nonempty sets and then provide {\it complete characterizations} of these properties in terms of the {\it relative contingent coderivatives}.

The major attention in our analysis is paid to the following {\it relative Lipschitz-like} property, which is the core of stability theory and applications.

\begin{definition}\label{lip-like} Let $S:X\rightrightarrows Y$ be a set-valued mapping between Banach spaces, $\Omega\subset X$, $\Theta \subset Y$ and $(\bar x,\bar y)\in\operatorname{gph}S|_\Omega^\Theta$. We say that $S$ has the {\sc Lipschitz-like property relative to} $\Omega \times \Theta$ around $(\bar x,\bar y)$ if $\operatorname{gph}S|_\Omega^\Theta$ is closed around this point and there exist neighborhoods $V\in{\cal N}(\bar x)$, $W\in{\cal N}(\bar y)$ and a constant $\kappa\ge 0$ such that\vspace*{-0.05in}
\begin{equation*}
S|^\Theta(x')\cap W \subset S|^\Theta(x) +\kappa\|x'-x\|_X B_Y\;\mbox{ for all }\;x,x'\in \Omega\cap V.
\end{equation*}
The {\sc Lipschitzian modulus} of $S$ {\sc relative to} $\Omega\times \Theta$ around $(\bar x,\bar y)$ is defined by
$$
\begin{aligned}
\operatorname{lip}_\Omega^\Theta S(\bar x,\bar y):=&\inf\big\{\kappa\geq0~\big|~\exists V\in\mathcal{N}(\bar x),\;W\in\mathcal{N}(\bar y)~\operatorname{with}\\
&S|^\Theta(x')\cap W \subset S|^\Theta(x) +\kappa\|x'-x\|_X B_Y\;\mbox{ for all }\;x,x'\in\Omega\cap V\big\}.
\end{aligned}
$$
\end{definition}\vspace*{-0.05in}
Take $x'=\bar x$, it is easy to check that $S|^\Theta(x)\cap W\not=\emptyset$ for all $x\in\Omega\cap V$.

The Lipschitz-like property relative to sets arises from some practical
optimization models. For example, the candidate point under consideration may be at the
boundary of the domain of set-valued mappings. On the other hand, in order to guarantee
some stationarity conditions one may only need a regular behavior of the constrained systems
with respect to one single critical direction, not on the whole space. This property is an extension of the well-known Lipschitz-like (or Aubin's pseudo-Lipschitz) property to constrained mappings, which corresponds to Definition~\ref{lip-like} with $\Omega=X$ and $\Theta=Y$ originally introduced in \cite{a84}; see also \cite{Mordukhovich2006,Rockafellar1998} for more details, references, and applications. The corresponding results for the latter property are restrictive for applications since they require in fact that the reference point belongs to the {\it interior} of the domain. A finite-dimensional version of Definition~\ref{lip-like} appeared in
\cite{Yang2021} in the case of convex sets $\Omega$ and $\Theta=R^m$ and was studied there by using the projectional coderivative. Our approach and results here via the 
relative contingent coderivatives are different from \cite{Yang2021} even in the case of finite-dimensional spaces.

We begin with deriving necessary conditions and sufficient conditions for the relative Lipschitz-like property that involve  points from {\it neighborhoods} of the reference one $(\bar x,\bar y)$. The first result provides a {\it neighborhood necessary condition} valid in arbitrary Banach spaces without the closedness assumptions on $\Omega$.\vspace*{-0.05in}

\begin{lemma}\label{necessity} Let $S:X\rightrightarrows Y$ be a set-valued mapping between Banach spaces $X$ and $Y$, $\Omega\subset X$, $\Theta\subset Y$, and $(\bar x,\bar y)\in\operatorname{gph}S|_\Omega^\Theta$. If $S$ has the Lipschitz-like property relative to $\Omega \times \Theta$ around $(\bar x,\bar y)$ with constant $\kappa\ge 0$, then there exists $\delta>0$ such that
\begin{equation}\label{nece}
\begin{array}{ll}
\qquad\qquad\qquad\|x^*\|_{X^*}\leq\kappa\|y^*\|_{Y^*}+\varepsilon(1+\kappa)\\
\;\mbox{for all }\;(x^*,-y^*)\in\widehat N_\varepsilon\big((x,y);\operatorname{gph} S|_\Omega^\Theta\big)\cap\big[J_X\big(T(x;\Omega)\big)\times J_Y\big(T(y;\Theta)\big)\big]
\end{array}
\end{equation}
whenever $\varepsilon\ge 0$ and $(x,y)\in\operatorname{gph} S|_\Omega^\Theta\cap\big[(\bar x+\delta B_X)\times(\bar y+\delta B_Y)\big]$.
\end{lemma}
\begin{proof}
When $\kappa=0$, there are $V\in\mathcal N(\bar x)$ and $W\in\mathcal N(\bar y)$ such that $S(x)\cap \Theta\cap W$ is a fixed set denoted by $E$ for any $x\in \Omega\cap V$.
Take any $\varepsilon\geq0$, $(x,y)\in V\times W$, and
$$
(x^*,-y^*)\in\widehat N_\varepsilon\big((x,y);\operatorname{gph} S|_\Omega^\Theta \big)\cap\big[J_X\big(T(x;\Omega)\big)\times J_Y\big(T(y;\Theta)\big)\big].
$$
As $S(x)\cap \Theta \cap W=E$ for any $x\in\Omega\cap V$, it is straightforward to check that
$$
\widehat N_\varepsilon\big((x,y);\operatorname{gph} S|_\Omega^\Theta\big)
=\widehat N_\varepsilon\big((x,y);\Omega\times E\big)
\subset\widehat N_\varepsilon(x;\Omega)\times\widehat N_\varepsilon(y;E),
$$
where the last inclusion follows from the proof of \cite[Proposition 1.2]{Mordukhovich2006}. Therefore, $x^*\in\widehat N_\varepsilon(x;\Omega)\cap J_X\big(T(x;\Omega)\big)$, which verifies $\|x^*\|_{X^*}\leq\varepsilon$ {  by Lemma \ref{JTN}}.

Suppose next that $\kappa>0$ and assume that there is $\delta>0$ such that
\begin{equation}\label{nece1}
S(x)\cap\Theta\cap(\bar y+\delta B_Y)\subset S(x')\cap\Theta+\kappa\|x-x'\|_X B_Y\;\mbox{ for all }\;x,x'\in \Omega\cap(\bar x+2\delta B_X).
\end{equation}
For any $\varepsilon\geq0$, $\gamma>0$, $(x,y)\in\operatorname{gph} S|_\Omega^\Theta\cap[(\bar x+\delta B_X)\times (\bar y+\delta B_Y)]$, and
\begin{equation*}
(x^*,-y^*)\in\widehat N_\varepsilon\big((x,y);\operatorname{gph} S|_\Omega^\Theta\big)\cap\big[J_X\big(T(x;\Omega)\big)\times J_Y\big(T(y;\Theta)\big)\big],
\end{equation*}
there exists a positive number $\alpha\leq\min\{\delta,\kappa\delta\}$ with
\begin{equation}\label{nece3}
\langle x^*,x'-x\rangle-\langle y^*,y'-y\rangle\leq(\varepsilon+\gamma)(\|x'-x\|_X+\|y'-y\|_Y)
\end{equation}
whenever $(x',y')\in \operatorname{gph} S|_\Omega^\Theta\cap[(x+\alpha B_X)\times(y+\alpha B_Y)]$. Moreover, it follows from $x^*\in J_X\big(T(x;\Omega)\big)$ that there are $z\in T(x;\Omega)$, $t_k\downarrow0$, and $x_k\stackrel{\Omega}{\longrightarrow}x$ such that
\begin{equation}\label{nece4}
x^*\in J_X(z)~~\operatorname{and}~~\frac{x_k-x}{t_k}\to z\;\mbox{ as }\;k\to\infty.
\end{equation}
Therefore, there is an integer $k_0$ ensuring the estimate $\|x_k-x\|_X<\frac{\alpha}{\kappa}\;\mbox{ for all }\;k>k_0$, which implies in turn that
$\|x_k-\bar x\|_X\leq\|x_k-x\|_X+\|x-\bar x\|_X\leq2\delta$ as $k > k_0$.
Thus we can employ the  Lipschitzian inclusion \eqref{nece1} with $y\in S(x)\cap\Theta\cap(\bar y+\delta B_Y)$ and the chosen element $x_k$. This way allows us to find $y_k\in S(x_k)\cap\Theta$ satisfying
$\|y_k-y\|_Y\leq\kappa\|x_k-x\|_X\leq\kappa\cdot\frac{\alpha}{\kappa}=\alpha,\;k > k_0.$
Substituting $(x_k,y_k)$ into \eqref{nece3} yields
\begin{equation}\label{nece8}
\begin{aligned}
\langle x^*,x_k-x\rangle\le&\|y^*\|_{Y^*}\|y_k-y\|_Y+(\varepsilon+\gamma)(\|x_k-x\|_X+\|y_k-y\|_Y)\\
\leq&\big(\kappa\|y^*\|_{Y^*}+(\varepsilon+\gamma)(1+\kappa)\big)\|x_k-x\|_X\;\mbox{ for all }\;k>k_0.
\end{aligned}
\end{equation}
If $z=0$, then by \eqref{nece4} we have $x^*=0$, and so \eqref{nece} holds automatically. If $z\ne 0$, then $x_k\ne x$ for all $k$ large enough. Using \eqref{nece4} in this case brings us to
\begin{equation*}
\lim\limits_{k\to\infty}\bigg\langle x^*,\frac{x_k-x}{\|x_k-x\|_X}\bigg\rangle
=\bigg\langle x^*,\frac{z}{\|z\|_X}\bigg\rangle=\|x^*\|_{X^*}.
\end{equation*}
This, combined with \eqref{nece8}, yields $\|x^*\|_{X^*}\leq\kappa\|y^*\|_{Y^*}+(\varepsilon+\gamma)(1+\kappa)$.
Since $\gamma>0$ was chosen arbitrarily, this verifies the claimed condition \eqref{nece}.
\end{proof}\vspace*{-0.05in}

Now we are ready to establish a {\it neighborhood sufficient condition} for the relative Lipschitz-like property of multifunctions in terms of $\varepsilon$-normals to graphs as $\varepsilon\ge 0$. The proof of this result requires the {\it reflexivity} of the Banach spaces in question due to the usage of the aforementioned properties of the duality mapping \eqref{dual map}.\vspace*{-0.05in}

\begin{lemma}\label{Sufficiency} Let $S:X\rightrightarrows Y$ be a set-valued mapping between reflexive Banach spaces, and let $\Omega\subset X$ and $\Theta\subset Y$ be  closed and convex with $\gph S|_\Omega^\Theta$ being closed around $(\bar x,\bar y)\in\operatorname{gph}S|_\Omega^\Theta$. Suppose that there exist numbers $\delta,\varepsilon_0>0$ and $\kappa \ge 0$ ensuring that for all $\varepsilon\in[0,\varepsilon_0)$ and $(x,y)\in\operatorname{gph} S|_\Omega^\Theta\cap\big[(\bar x+\delta B_X)\times(\bar y+\delta B_Y)\big]$ we have the dual norm estimate
\begin{equation}\label{suff}
\begin{array}{ll}
\qquad\qquad\qquad\|x^*\|_{X^*}\leq\kappa\|y^*\|_{Y^*}+\sqrt\varepsilon(1+\kappa)\\
\mbox{whenever }\;(x^*,-y^*)\in\widehat N_\varepsilon\big((x,y);\operatorname{gph} S|_\Omega^\Theta\big)\cap\big[J_X\big(T(x;\Omega)\big)\times J_Y\big(T(y;\Theta)\big)\big].
\end{array}
\end{equation}
Then $S$ is Lipschitz-like relative to $\Omega\times\Theta$ around $(\bar x,\bar y)$ with modulus $\kappa$.
\end{lemma}
\begin{proof}
Assume without loss of generality that the set $\operatorname{gph}S|_\Omega^\Theta$ is closed. Suppose on the contrary that $S$ does not have the Lipschitz-like property relative to $\Omega\times \Theta$ around $(\bar x,\bar y)$ with constant $\kappa$. Then for any $\mu\in(0,\delta/5(\kappa+1))$, there exist $x',x''\in\Omega\cap[\bar x+\mu B_X]$ with $x'\not=x''$ and $y''\in S(x'')\cap \Theta\cap[\bar y+\mu B_Y]$ such that $d(y'',S(x')\cap\Theta)>\kappa\|x'-x''\|_X$. This allows us to find $\kappa'\in(\kappa,\kappa+1)$ such that
\begin{equation}\label{lem3.3a}
d(y'',S(x')\cap\Theta)>\kappa'\|x'-x''\|_X.
\end{equation}
Take $\varepsilon\leq(\kappa'-\kappa)^2/4(1+\kappa)^2$. Claiming that \eqref{suff} fails will provide a contradiction.
To proceed, denote $\sigma:=(\kappa'-\kappa)\|x'-x''\|_X/4$ and define
$$
\varphi(x,y):=\frac{\kappa'+\kappa}{2}\|x-x'\|_X+\sqrt{\|y-y''\|_Y^2+\sigma^2}+\iota_{\gph S|_\Omega^\Theta}(x,y),~~~(x,y)\in X\times Y.
$$
By the closedness of \text{gph}~$S|_\Omega^\Theta$ and $(x'',y'') \in \operatorname{gph} S|_\Omega^\Theta$, the function $\varphi$ is proper, l.s.c. and nonnegative on $X\times Y$ with $\inf_{X\times Y}\varphi<\infty$. For $\eta:=\min\{\sigma,\varepsilon\}/2$, there is $(x_\eta,y_\eta)\in X\times Y$ such that $\varphi(x_\eta,y_\eta)\leq\inf_{X\times Y} \varphi+\eta$. Applying the Ekeland variational principle (see, e.g., \cite[Theorem~2.26]{Mordukhovich2006}) gives us a pair $(\tilde{x},\tilde{y})\in X\times Y$ such that
\begin{eqnarray}
\varphi(\tilde{x},\tilde{y}) \leq \varphi(x,y)+\eta\|(x,y)-(\tilde x,\tilde y)\|_{X\times Y}\;\mbox{ for all }\;(x,y)\in X\times Y.\label{align2}
\end{eqnarray}
Define $\phi(x,y):=\eta\|(x,y)-(\tilde x,\tilde y)\|_{X\times Y}$ and observe that $\phi$ is Lipschitz continuous on $X \times Y$. { It follows from \eqref{align2} that $(\tilde{x},\tilde{y})$ is a local minimum of $\varphi + \phi$. Combining this with \cite[Lemma 2.32(i)]{Mordukhovich2006}  ensures the existence of $(x_i,y_i)\in\operatorname{gph} S|_\Omega^\Theta\cap[(\tilde x+\eta B_X)\times (\tilde y+\eta B_Y)]$, $i=1,2$, 
satisfying the inequalities}
\begin{equation}\label{s1}
\varphi(x_i,y_i)\leq\varphi(x_i,y_i)+\phi(x_i,y_i)\leq\varphi(\tilde x,\tilde y)+\phi(\tilde x,\tilde y)+\eta,~~i=1,2,
\end{equation}
together with the subgradient inclusion
\begin{equation}\label{s2}
(0,0)\in\hat \partial\varphi(x_1,y_1)+\hat\partial\phi(x_2,y_2)+\eta B_{X^*\times Y^*}.
\end{equation}
It follows from \eqref{s1} that $\varphi(x_1,y_1)<\infty$, which implies in turn that $(x_1,y_1)\in\gph S|_\Omega^\Theta$. Moreover, \eqref{s1} also tells us that
\begin{equation}\label{s3}
\begin{aligned}
\varphi(x_1,y_1)\leq&\varphi(\tilde x,\tilde y)+\eta\leq\inf_{X\times Y}\varphi+2\eta\leq\varphi(x'',y'')+2\eta\\
=&\frac{\kappa'+\kappa}{2}\|x''-x'\|_X+\sigma+2\eta\leq\kappa'\|x''-x'\|_X.
\end{aligned}
\end{equation}
Observe furthermore that
\begin{equation}\label{s4}
\varphi(x_1,y_1)
\geq\frac{\kappa'+\kappa}{2}
\|x_1-x'\|_X+\|y_1-y''\|_Y
\end{equation}
which allows us deduce from \eqref{s3} and \eqref{s4} the estimates
$$
\begin{aligned}
\|x_1-\bar x\|_X\leq&\|x_1-x'\|_X+\|x'-\bar x\|_X\leq\frac{2}{\kappa'+\kappa}\varphi(x_1,y_1)+\mu\\
\leq&\frac{2\kappa'}{\kappa'+\kappa}\|x''-x'\|_X+\mu\leq5\mu<\delta,
\end{aligned}
$$
$$
\begin{aligned}
\|y_1-\bar y\|_Y\leq&\|y_1-y''\|_Y+\|y''-\bar y\|_Y\leq\varphi(x_1,y_1)+\mu \\
\leq&\kappa'\|x''-x'\|_X+\mu\leq(2\kappa+3)\mu<\delta.
\end{aligned}
$$
and verifies therefore that $(x_1,y_1)\in\gph S|_\Omega^\Theta\cap[(\bar x+\delta B_X)\times(\bar y+\delta B_Y)]$. Next we claim that $x_1\not=x'$. Indeed, supposing $x_1=x'$ gives us by \eqref{lem3.3a}, \eqref{s3}, and \eqref{s4} that
$$
\begin{aligned}
\kappa'\|x'-x''\|_X<&d(y'',S(x')\cap\Theta)=d(y'',S(x_1)\cap\Theta)\leq\|y_1-y''\|_Y\\
\leq&\varphi(x_1,y_1)\leq\kappa'\|x''-x'\|_X,
\end{aligned}
$$
which is a contradiction. Then we know that
$$
\hat\partial\varphi(x_1,y_1)=\left\{\left(\frac{(\kappa'+\kappa)J_X(x_1-x')}{2\|x_1-x'\|_X},
\frac{J_Y(y_1-y'')}{\sqrt{\|y_1-y''\|_Y^2+\sigma^2}}\right)\right\}+\widehat N((x_1,y_1);\gph S|_\Omega^\Theta)
$$
and notice that $\hat\partial\phi(x_2,y_2)\subset\eta B_{X^*\times Y^*}$. Taking
$$
x^*:=\frac{(\kappa'+\kappa)J_X(x'-x_1)}{2\|x_1-x'\|_X}~~\mbox{and}~~y^*:=\frac{J_Y(y_1-y'')}{\sqrt{\|y_1-y''\|_Y^2+\sigma^2}},
$$
we get $(x^*,-y^*)\in J_X\big(T(x_1;\Omega)\big)\times J_Y\big(T(y_1;\Theta)\big)$ by the facts that $x' \in \Omega$, $y'' \in \Theta$, homogeneity of $J$, and the sets $\Omega$ and $\Theta$ are convex. Moreover, it follows from \eqref{s2} that
$(x^*,-y^*)\in\; \widehat N((x_1,y_1);\gph S|_\Omega^\Theta)+2\eta B_{X^*\times Y^*}$
$\subset\; \widehat N((x_1,y_1);\gph S|_\Omega^\Theta)+\varepsilon B_{X^*\times Y^*}\subset\widehat N_\varepsilon((x_1,y_1);\gph S|_\Omega^\Theta).$ Due to the estimates
$$
\|x^*\|_{X^*}=\frac{\kappa'+\kappa}{2}=\kappa+\frac{\kappa'-\kappa}{2}>\kappa\|y^*\|_{Y^*}
+\frac{\kappa'-\kappa}{2(\kappa+1)}(\kappa+1)\geq\kappa\|y^*\|_{Y^*}+\sqrt\varepsilon(\kappa+1),
$$
we arrive at a  contradiction with  \eqref{suff} and completes the proof of the lemma.
\end{proof}\vspace*{-0.02in}

Now we show that if both duality mappings of reflexive Banach spaces $X$ and $Y$ are uniformly continuous on their unit spheres, then the same conclusion as in Lemma~\ref{Sufficiency} holds when $\Omega$ and $\Theta$ are $C^2$-manifolds. Recall that a set $\Omega$ is a {\it $C^m$-manifold} ($i=1,2$) if for any $x\in\Omega$ there exist $U\in\mathcal N(x)$ and a $C^m$-homeomorphism $\psi:X_\psi\to\Omega\cap U$ for some Banach space $X_\psi$, where $\psi$ is a {\it $C^m$-homeomorphism} if both $\psi$ and $\psi^{-1}$ are of class $C^m$. It is easy to show that for such $\Omega$ and $\bar x\in\Omega$ with $\psi(0)=\bar x$, we have the relationships
\begin{eqnarray}
&&T(\bar x;\Omega)=\{\nabla\psi(0)\xi~|~\xi\in X_\psi\} \label{tangent}\;\mbox{ and }\\
&&w^*\mbox{-}\limsup\limits_{x\stackrel{\Omega}{\longrightarrow}\bar x}J_X\big(T(x;\Omega)\big)\subset \co J_X\big(T(\bar x;\Omega)\big).\label{tangent-limit}
\end{eqnarray}

\begin{lemma}\label{mani}
Let $X$ be a reflexive Banach space with its duality mapping $J_X$ being uniformly continuous on the unit sphere $\mathbb S_X$, and let $\Omega\subset X$. Given $\bar x\in \Omega$, assume that there exist $U\in\mathcal N(\bar x)$ and a $C^2$-homeomorphism $\psi:X_\psi\to \Omega \cap U$ with  a Banach space $X_\psi$ and $\psi(0)=\bar x$. Then for any $\varepsilon>0$, we find $\delta>0$ such that
$$
d\left(\frac{J_X(x'-x)}{\|x'-x\|_X},J_X(T(x;\Omega)\cap\mathbb S_X)\right)<\varepsilon\;\mbox{ whenever }\;x,x'\in\Omega\cap[\bar x+\delta B_X]~\mbox{with}~x\ne x'.
$$
\end{lemma}
\begin{proof}
Let $\varepsilon>0$. The uniformly continuity and homogeneity of $J_X$ on the unit sphere $\mathbb S_X$ guarantee that there exists $\varrho\in(0,1)$ such that
\begin{equation}\label{mani1}
\|J_X(x')-J_X(x)\|_{X^*}\leq\varepsilon\;\mbox{ for all }\;x,x'\in\mathbb S_X~\mbox{with}~\|x-x'\|_X\leq\varrho.
\end{equation}
Since $\psi$ is an $C^2$-homeomorphism, we can find positive numbers $\eta,L,\sigma$ with
\begin{eqnarray}
&&  \|\psi^{-1}(x)-\psi^{-1}(x')\|_{X_\psi}\leq L\|x-x'\|_X\;\mbox{ if }\;x,x'\in\Omega\cap[\bar x+\eta B_X],\label{mani2}\\
&&\|\psi(\xi')-\psi(\xi)-\nabla\psi(\xi)(\xi'-\xi)\|_{X}\leq\frac{\varrho}{2L}\|\xi-\xi'\|_{X_{\psi}}\;\mbox{ if }\;\xi,\xi'\in\sigma B_{X_\psi}.\label{mani3}
\end{eqnarray}
Taking $\delta:=\min\{\eta,\sigma/L\}$ and $x,x'\in\Omega\cap[\bar x+\delta B_X]$ with $x\ne x'$, let $\xi:=\psi^{-1}(x)$ and $\xi':=\psi^{-1}(x')$. Then it follows from \eqref{mani2} that $\xi,\xi'\in\sigma B_{X_\psi}$ with $\xi\ne\xi'$. Employing \eqref{mani2} gives us the relationships
\begin{equation}\label{mani4}
\|\psi(\xi)-\psi(\xi')\|_X=\|x-x'\|_X\geq\frac{1}{L}\|\psi^{-1}(x)-\psi^{-1}(x')\|_{X_{\psi}}
=\frac{1}{L}\|\xi-\xi'\|_{X_\psi}>0,
\end{equation}
which being combined with \eqref{mani3} yield
\begin{equation}\label{mani4a}
\|\nabla\psi(\xi)(\xi'-\xi)\|_X\geq\left(\frac{1}{L}-\frac{\varrho}{2L}\right)\|\xi-\xi'\|_{X_\psi}
\geq\frac{1}{2L}\|\xi-\xi'\|_{X_\psi}>0.
\end{equation}
It follows from \eqref{mani3} and \eqref{mani4} 
$$
\left\|\frac{\psi(\xi')-\psi(\xi)}{\|\psi(\xi')-\psi(\xi)\|_X}
-\frac{\nabla\psi(\xi)(\xi'-\xi)}{\|\psi(\xi')-\psi(\xi)\|_X}\right\|_X\leq\frac{\varrho}{2}.
$$
Moreover, \eqref{mani3}, \eqref{mani4} and \eqref{mani4a} tell us that
$$
\left\|\frac{\nabla\psi(\xi)(\xi'-\xi)}{\|\psi(\xi')-\psi(\xi)\|_X}
-\frac{\nabla\psi(\xi)(\xi'-\xi)}{\|\nabla\psi(\xi)(\xi'-\xi)\|_X}\right\|_X\leq\frac{\varrho}{2}.
$$
Therefore, we arrive at the estimate
$$
\left\|\frac{x'-x}{\|x'-x\|_X}
-\frac{\nabla\psi(\xi)(\xi'-\xi)}{\|\nabla\psi(\xi)(\xi'-\xi)\|_X}\right\|_X\leq\varrho.
$$
Unifying the latter with \eqref{tangent} and \eqref{mani1} results in 
$$
d\left(\frac{J_X(x'-x)}{\|x'-x\|_X},J_X(T(x;\Omega)\cap\mathbb S_X)\right)<\varepsilon,
$$
which completes the proof of the lemma.
\end{proof}

We know from \cite[Theorem 1.2]{Barbu1993} that the duality mapping $J_X$ of a reflexive Banach space $X$ with the uniformly convex dual space $X^*$ is uniformly continuous on $\mathbb S_X$. 
The next lemma provides a {\it neighborhood sufficient condition} for the {\it relative Lipschitz-like property} of a set-valued mapping on {\it manifolds}. Its proof is similar to that of Lemma \ref{Sufficiency} by applying Lemma~\ref{mani}, the Ekeland variational principle and the subgradient description of the extremal principle from \cite[Lemma 2.32(i)]{Mordukhovich2006}. Thus it is omitted.\vspace*{-0.02in}

\begin{lemma}\label{Sufficiency'} Let $S:X\rightrightarrows Y$ be a set-valued mapping between reflexive Banach spaces, and let $\Omega\subset X$ and $\Theta\subset Y$ be $C^2$-manifolds such that $\gph S|_\Omega^\Theta$ is closed around $(\bar x,\bar y)\in\operatorname{gph}S|_\Omega^\Theta$. Suppose that both $X^*$ and $Y^*$ are uniformly convex, and that there exist numbers $\delta,\varepsilon_0>0$ and $\kappa \ge 0$ ensuring that for all $\varepsilon\in[0,\varepsilon_0)$ and $(x,y)\in\operatorname{gph} S|_\Omega^\Theta\cap\big[(\bar x+\delta B_X)\times(\bar y+\delta B_Y)\big]$ we have
\begin{equation}\label{suff'}
\begin{array}{ll}
\|x^*\|_{X^*}\leq\kappa\|y^*\|_{Y^*}+\sqrt\varepsilon(1+\kappa)\;\mbox{ whenever }\\
(x^*,-y^*)\in\widehat N_\varepsilon\big((x,y);\operatorname{gph} S|_\Omega^\Theta\big)\cap\big[J_X\big(T(x;\Omega)\big)\times J_Y\big(T(y;\Theta)\big)\big].
\end{array}
\end{equation}
Then $S$ is Lipschitz-like relative to $\Omega\times\Theta$ around $(\bar x,\bar y)$ with any constant $\kappa'>\kappa$.
\end{lemma}\vspace*{-0.02in}

Our main goal here is to derive {\it necessary and sufficient conditions} for the relative Lipschitz-like property expressed, in contrast to the neighborhood results of Lemmas~\ref{necessity} and \ref{Sufficiency}, via generalized differential constructions depending only on the {\it point in question}. To establish such {\it pointbased} characterizations, we  define the following {\it relative partial sequential normal compactness} condition (relative PSNC), which extends the standard PSNC one from \cite{Mordukhovich2006} to the case of constrained multifunctions.\vspace*{-0.02in}

\begin{definition}\label{psnc} Let $S:X\rightrightarrows Y$ be a multifunction between Banach spaces, $\Omega\subset X$, and $\Theta\subset Y$. We say that $S$ is {\sc PSNC relative to} $\Omega\times\Theta$ at $(\bar x,\bar y)\in\operatorname{gph}S|_\Omega^\Theta$ if for any sequence $(\varepsilon_k,x_k,y_k,x_k^*,y_k^*)\in (0,1)\times\operatorname{gph} S|_\Omega^\Theta\times X^*\times Y^*$ satisfying
\begin{equation}\label{PSNC}
\begin{aligned}
&\varepsilon_k\downarrow0,\;(x_k,y_k)\to(\bar x,\bar y),\;x_k^*\rightharpoonup^*0,\;\operatorname{and}~y_k^*\to 0\;\mbox{ as }\;k\to\infty,\;\mbox{ and}\\
&(x_k^*,-y_k^*)\in \widehat N_{\varepsilon_k}\big((x_k,y_k);\operatorname{gph} S|_\Omega^\Theta\big)\cap[J_X\big(T(x_k;\Omega)\big)\times J_Y\big(T(y_k;\Theta)\big)],
\end{aligned}
\end{equation}
the strong $X^*$-convergence $\|x_k^*\|_{X^*}\to 0$ as $k\to \infty$ holds. In particular, we say $S$ is PSNC relative to $\Omega$ at $(\bar x,\bar y)$ when $\Theta=Y$.

In addition, we say that $S$ is {\sc mirror PSNC relative to} $\Omega\times\Theta$ at $(\bar x,\bar y)\in\operatorname{gph}S|_\Omega^\Theta$ if for any sequence $(\varepsilon_k,x_k,y_k,x_k^*,y_k^*)\in (0,1)\times\operatorname{gph} S|_\Omega^\Theta\times X^*\times Y^*$ satisfying (\ref{PSNC}) with $x_k^*\rightharpoonup^*0,\;\operatorname{and}~y_k^*\to 0$ being replaced by $x_k^*\to0, \mbox{ and } y_k^*\rightharpoonup^* 0,$
the strong $Y^*$-convergence $\|y_k^*\|_{Y^*}\to 0$ as $k\to \infty$ holds. In particular, we say $S$ is mirror PSNC relative to $\Omega$ at $(\bar x,\bar y)$ when $\Theta=Y$.
\end{definition}\vspace*{-0.03in}

We obviously have that the PSNC property of $S|_\Omega^\Theta$ ensures that $S$ is PSNC relative to $\Omega\times \Theta$. However, the converse can fail in general. In addition, from the definitions of the mixed contingent coderivative and the relative PSNC property, we can directly deduce the conclusions of the next lemma.\vspace*{-0.02in}

\begin{lemma}\label{mirror}
 Let $S:X\rightrightarrows Y$ be a multifunction between Banach spaces, $\Omega\subset X$, and $\Theta\subset Y$. Given $(\bar x,\bar y)\in\gph S|_\Omega^\Theta$, the following assertions are satisfied:

{\bf (i)} $S^{-1}$ is PSNC relative to $\Theta\times\Omega$ at $(\bar y,\bar x)$ if and only if $S$ is mirror PSNC relative to $\Omega\times\Theta$ at $(\bar x,\bar y)$.

{\bf (ii)} For any $x^*\in X^*$ and $y^*\in Y^*$, we have the equivalence
$$
x^*\in D_m^*S_\Omega^\Theta(\bar x\vbl\bar y)(y^*)~\Longleftrightarrow~-y^*\in D_M^*(S^{-1})_\Theta^\Omega(\bar y\vbl\bar x)(-x^*).
$$
\end{lemma}

If  $S$ is strictly differentiable at $\bar x$, it is Lipschitz-like around that point, { and hence $S$ is PSNC therein; see \cite[Theorem 4.10]{Mordukhovich2006}. Definition~\ref{psnc} immediately implies that the relative PSNC property is automatic if dim $X<\infty$. Observe that when $\bar x\in{\rm int}\,\Omega$ and $\bar y\in{\rm int}\,\Theta$, the intersection $\widehat N_{\varepsilon_k}\big((x_k,\hskip-0.05cm u_k);\operatorname{gph} S|_\Omega^\Theta\big)\cap[J_X\big(T(x_k;\Omega)\big)\times J_Y\big(T(y_k;\Theta)\big)]$ reduces to $\widehat N_{\varepsilon_k}\big((x_k,\hskip-0.05cm  u_k);\operatorname{gph} S\big)$, and thus the relative PSNC agrees in this case with the standard PSNC from \cite{Mordukhovich2006} where the set $\Omega$ and $\Theta$ are not present at all. It will be shown in Section~\ref{sec:calculus} that the newly introduced relative PSNC, similarly to the standard PSNC in \cite{Mordukhovich2006}, {\it is preserved} under major operations over multifunctions. Note to this end that the essentially more restrictive SNC property for sets and set-valued mappings (not considered in this paper) may not hold for some convex sets in Sobolev spaces important in applications to PDE control problems; see \cite{Mehlitz2019} for more details}.\vspace*{0.05in}

Here is the main result of this section providing {\it pointbased characterizations} of the relative Lipschitz-like property and the evaluation of its {\it exact Lipschitzian bound}.\vspace*{-0.03in}

\begin{theorem}\label{criterion} Let $S:X\rightrightarrows Y$ be a set-valued mapping between reflexive Banach spaces $X$ and $Y$, and let $\Omega$ and $\Theta$ be a closed and convex subset of $X$ and $Y$, respectively. Suppose that $(\bar x,\bar y)\in\operatorname{gph} S|_\Omega^\Theta$ and the set $\operatorname{gph} S|_\Omega^\Theta$ is closed around $(\bar x,\bar y)$. Then the following assertions are equivalent:

{\bf(i)} $S$ is Lipschitz-like relative to $\Omega \times \Theta$ around $(\bar x,\bar y)$.

{\bf(ii)} $S$ is PSNC relative to $\Omega \times \Theta$ at $(\bar x,\bar y)$, and $|D_M^*S_\Omega^\Theta(\bar x\vbl\bar y)|^+<\infty$.

{\bf(iii)} $S$ is PSNC relative to $\Omega\times\Theta$ at $(\bar x,\bar y)$, and $D_M^*S_\Omega^\Theta(\bar x\vbl\bar y)(0)=\{0\}$.\\[0.5ex]
Moreover, under the fulfillment of {\rm(i)--(iii)}, the exact Lipschitzian bound of $S$ relative to $\Omega$ around $(\bar x,\bar y)$ satisfies the estimates
\begin{equation}\label{inequ-estimate}
|D_M^*S_\Omega^\Theta(\bar x\vbl\bar y)|^+\leq \operatorname{lip}_\Omega^\Theta S(\bar x,\bar y)\leq |D_N^*S_\Omega^\Theta(\bar x\vbl\bar y)|^+,
\end{equation}
where the upper estimate holds when $\operatorname{dim} X<\infty$. If in addition the mapping $S$ is coderivatively normal relative to $\Omega \times \Theta$ at $(\bar x,\bar y)$, then we have
\begin{equation}\label{equ-estimate}
\operatorname{lip}_\Omega^\Theta S(\bar x,\bar y)=|D_M^*S_\Omega^\Theta(\bar x\vbl\bar y)|^+=  |D_N^*S_\Omega^\Theta(\bar x\vbl\bar y)|^+.
\end{equation}
\end{theorem}
\begin{proof}
{\bf[(i)$\Longrightarrow$(ii)]:} Let $S$ be Lipschitz-like property relative to $\Omega\times \Theta$ around $(\bar x,\bar y)$ with constant $\kappa$. Take any sequence $(\varepsilon_k,x_k,y_k,x_k^*,y_k^*)\in (0,1)\times\operatorname{gph} S|_\Omega^\Theta\times X^*\times Y^*$ satisfying \eqref{PSNC} and deduce from Lemma~\ref{necessity} that
\begin{equation*}
\|x_k^*\|_{X^*}\le\kappa\|y_k^*\|_{Y^*}+\varepsilon_k(1+\kappa)\;\mbox{  for sufficiently large }\;k.
\end{equation*}
Thus $x_k^*\to 0$ as $k\to\infty$, which shows that $S$ is PSNC relative to $\Omega\times\Theta$ at $(\bar x,\bar y)$.

To verify the second assertion in {\bf (ii)}, take any $(x^*, y^*)\in X^*\times Y^*$ with $x^*\in D_M^*S_\Omega^\Theta(\bar x\vbl\bar y)(y^*)$. Then there exist sequences $\varepsilon_k\downarrow0$, $(x_k,y_k)\stackrel{\operatorname{gph} S|_\Omega^\Theta}{\longrightarrow}(\bar x,\bar y)$, and
\begin{equation*}
(x_k^*,-y_k^*)\in\widehat N_{\varepsilon_k}\big((x_k,y_k);\operatorname{gph} S|_\Omega^\Theta\big)\cap\big[J_X\big(T(x_k;\Omega)\big)\times Y^*\big],\quad k\in\mathbb N,
\end{equation*}
such that $y_k^*\to y^*$ and $x_k^*\rightharpoonup^* x^*$. It follows from Lemma~\ref{necessity} that $\|x_k^*\|_{X^*}\le\kappa\|y_k^*\|_{Y^*}+\varepsilon_k(1+\kappa)$ for all $k$ sufficiently large. Thus we have
$$
\|x^*\|_{X^*}\leq\liminf\limits_{k\to\infty}\|x_k^*\|_{X^*}
\leq\liminf\limits_{k\to\infty}\big(\kappa\|y_k^*\|_{Y^*}+\varepsilon_k(1+\kappa)\big)
=\kappa\|y^*\|_{Y^*},
$$
and hence $|D_M^*S_\Omega^\Theta(\bar x\vbl\bar y)|^+\leq\kappa$. Note that the obtained implication as well as the lower estimate in \eqref{inequ-estimate} hold in any Banach spaces $X$ and $Y$.

{\bf[(ii)$\Longrightarrow$(iii)]}: This is obvious since $D_M^*S_\Omega^\Theta(\bar x\vbl\bar y)$ is positively homogeneous.

{\bf[(iii)$\Longrightarrow$(i)]}: Suppose that $S$ is not Lipschitz-like relative to $\Omega\times\Theta$ around $(\bar x,\bar y)$. By Lemma~\ref{Sufficiency}, there exist sequences $\varepsilon_k\downarrow0$, $(x_k,y_k)\stackrel{\operatorname{gph} S|_\Omega^\Theta}{\longrightarrow}(\bar x,\bar y)$, and $(x_k^*,-y_k^*)\in\widehat N_{\varepsilon_k}\big((x_k,y_k);\operatorname{gph} S|_\Omega^\Theta\big)\cap\big[J_X\big(T(x_k;\Omega)\big)\times J_Y\big(T(y_k;\Theta)\big)\big]$ such that
\begin{equation}\label{mainp1}
\|x^*_k\|_{X^*}>k\|y_k^*\|_{Y^*}+\sqrt{\varepsilon_k}(1+k)\;\mbox{ whenever }\;k\in\mathbb N,
\end{equation}
which ensures that $\|x^*_k\|_{X^*}\ne 0$ for all $k$. Define the modified sequences
\begin{equation}\label{crit0}
\hat x_k^*:=\frac{x_k^*}{\|x_k^*\|_{X^*}},~~\hat y_k^*=\frac{y_k^*}{\|x_k^*\|_{X^*}},~~\operatorname{and}~~\hat\varepsilon_k=\frac{\sqrt{\varepsilon_k}}{\|x_k^*\|_{X^*}}
\end{equation}
with  $\|\hat x_k^*\|_{X^*}=1$ for all $k$. By \eqref{mainp1}, we have $\hat y_k^*\to0$ and  $\hat \varepsilon_k\to0$. Furthermore, it follows from the positive homogeneity of $\widehat N_\varepsilon(\bar x;\Omega)$ with respect to $\epsilon$ that
$$
(\hat x_k^*,-\hat y_k^*)\in\widehat N_{\hat \varepsilon_k}\big((x_k,y_k);\operatorname{gph} S|_\Omega^\Theta\big)\cap\big[J_X\big(T(x_k;\Omega)\big)\times J_Y\big(T(y_k;\Theta)\big)\big],\quad k\in\mathbb N.
$$
By passing to a subsequence if necessary, we assume that there is $x^*\in B_{X^*}$ such that $\hat x_k^*\rightharpoonup^* x^*$. Then the imposed coderivative condition $D_M^*S_\Omega^\Theta(\bar x\vbl\bar y)(0)=\{0\}$ tells us that $x^*=0$, i.e., $\hat x_k^*\rightharpoonup^*0$. Since $S$ is PSNC at $(\bar x,\bar y)$ relative to $\Omega\times\Theta$, we get that $\hat x_k^*\to 0$ as $k\to\infty$, which contradicts $\|\hat x_k^*\|_{X^*}=1$ and thus verifies (i).

It is clear that the lower estimate in \eqref{inequ-estimate} holds in any Banach spaces. Let us now justify that the upper one therein is satisfied if $Y$ is reflexive while $X$ is finite-dimensional. Fixing any $\kappa>|D_N^*S_\Omega^\Theta(\bar x\vbl\bar y)|^+$, we claim that there exists $\delta>0$ such that \eqref{suff} holds for any $\varepsilon>0$ sufficiently small. Indeed, the contrary means that there exist sequences $\varepsilon_k\downarrow0$, $(x_k,y_k)\stackrel{\operatorname{gph} S|_\Omega^\Theta}{\longrightarrow}(\bar x,\bar y)$ as $k\to\infty$, and
$$
(x_k^*,y_k^*)\in\widehat N_{\varepsilon_k}\big((x_k,y_k);\operatorname{gph} S|_\Omega^\Theta\big)\cap\big[J_X\big(T(x_k;\Omega)\big)\times J_Y\big(T(y_k;\Theta)\big)\big]
$$
along which we have the estimate
\begin{equation}\label{crit1}
\|x_k^*\|_{X^*}>\kappa\|y_k^*\|_{Y_*}+\sqrt{\varepsilon_k}(1+\kappa)\;\mbox{ for all }\;k.
\end{equation}
Defining $\hat x_k^*$, $\hat y_k^*$, and $\hat \varepsilon_k$ as in \eqref{crit0} gives us $\|\hat x_k^*\|_{X^*}=1$, $\|\hat y_k^*\|_{Y^*}\le\kappa^{-1}$, and
$$
(\hat x_k^*,-\hat y_k^*)\in\widehat N_{\hat \varepsilon_k}\big((x_k,y_k);\operatorname{gph} S|_\Omega^\Theta\big)\cap\big[J_X\big(T(x_k;\Omega)\big)\times J_Y\big(T(y_k;\Theta)\big)\big].
$$
Furthermore, it follows from \eqref{crit1} that
$$
\hat\varepsilon_k\leq\frac{\varepsilon_k}{\sqrt{\varepsilon_k}(1+\kappa)}
\leq\frac{\sqrt{\varepsilon_k}}{1+\kappa}\to 0\;\mbox{ as }\;k\to\infty.
$$
Since dim $X<\infty$ and $Y$ is reflexive, there exists $(x^*,y^*)\in X^*\times Y^*$ for which $\hat x_k^*\to x^*$ and $\hat y_k^*\rightharpoonup^*y^*$ along some subsequences. Then $x^*\in D_N^*S_\Omega^\Theta(\bar x\vbl\bar y)(y^*)$ and
$$
\|y^*\|_{Y^*}\leq\liminf\limits_{k\to\infty}\|y_k^*\|_{Y^*}
\leq\frac{1}{\kappa}=\frac{1}{\kappa}\|x^*\|_{X^*}
$$
due to the continuity of the norm function in finite-dimensional spaces and its lower
semicontinuity in the weak$^*$ topology of $Y^*$ for any reflexive Banach space. This contradicts the condition $\kappa>|D_N^*S_\Omega^\Theta(\bar x\vbl\bar y)|^+$ and thus verifies the upper estimate in \eqref{inequ-estimate}. If $S$ is coderivatively normal relative to $\Omega\times\Theta$ at $(\bar x,\bar y)$, the equalities in \eqref{equ-estimate} follow directly from the definition and the estimates in \eqref{inequ-estimate}.
\end{proof}\vspace*{-0.02in}

Applying Lemma \ref{Sufficiency'} confirms that the results of Theorem~\ref{criterion} hold true whenever {\it both sets $\Omega$ and $\Theta$ are $C^2$-manifolds, $X^*$ and $Y^*$ are uniformly convex}. This allows us to derive the following consequence of the theorem. Observe first that for any mapping $S:X\rightrightarrows Y$ between two reflexive Banach spaces and any $(\bar x,\bar y)\in\gph S|_\Omega^\Theta$ with $\Omega\subset X$ and $\Theta\subset Y$, we have from \eqref{in} and the definition of the mixed contingent coderivative that
$$
D_M^*S_\Omega^\Theta(\bar x\vbl\bar y)(y^*)\subset D_M^*S|_\Omega^\Theta(\bar x\vbl\bar y)(y^*),~~~\forall y^*\in Y^*.
$$
If moreover $\Omega$ is a $C^1$-manifold, then it follows from \eqref{tangent-limit} that
$$
D_M^*S_\Omega^\Theta(\bar x\vbl\bar y)(0)\subset\co J_X\big(T(\bar x;\Omega)\big).
$$
Having this in hand, we arrive at the following upper estimate of the {\it mixed contingent coderivative} via the usual coderivative of the {\it truncated restricted} mapping and the convex full of the {\it tangent cone image} under the duality mapping.\vspace*{-0.02in}

\begin{corollary}\label{upper-suff}
Let $X$ and $Y$ be reflexive Banach spaces whose dual spaces $X^*$ and $Y^*$ are uniformly convex. Given a set-valued mapping $S:X\rightrightarrows Y$ and sets $\Omega\subset X$ and $\Theta\subset Y$ that both are $C^2$-manifolds, assume that 
$\gph S|_\Omega^\Theta$ is closed around  $(\bar x,\bar y)\in\gph S|_\Omega^\Theta$. Then we have the inclusion
 $$
D_M^*S_\Omega^\Theta(\bar x\vbl\bar y)(0)\subset D_M^*S|_\Omega^\Theta(\bar x\vbl\bar y)(0)\cap\co J_X\big(T(\bar x;\Omega)\big).
$$
If furthermore $S$ is PSNC relative to $\Omega\times \Theta$ at $(\bar x,\bar y)$ and  
$$
D_M^*S|_\Omega^\Theta(\bar x\vbl\bar y)(0)\cap\co J_X\big(T(\bar x;\Omega)\big)=\{0\},
$$
then $S$ is Lipschitz-like relative to $\Omega\times \Theta$ around $(\bar x,\bar y)$.
\end{corollary}

In the case where $\dim X\cdot\dim Y<\infty$ and $\Omega=X$, $\Theta=Y$, Theorem~\ref{criterion} reduces to the characterization $D_M^*S(\bar x,\bar y)(0)=\{0\}$ of the Lipschitz-like property of $S$ with the formula $\operatorname{lip}S(\bar x,\bar y)=|D_M^*S(\bar x\vbl\bar y)|^+$ for the exact Lipschitzian bound expressed in terms of the basic coderivative from \cite{m80}; see \cite[Theorem~5.7]{m93} and also \cite[Theorem~9.40]{Rockafellar1998}, where this result is called the {\it Mordukhovich criterion}. In \cite{Yang2021}, this criterion is extended to constrained multifunctions between finite-dimensional spaces with convex sets $\Omega$ in terms of the projectional coderivative different from the contingent one. For multifunctions between infinite-dimensional spaces, an unconstrained version of Theorem~\ref{criterion} is given in \cite[Theorem~4.9]{Mordukhovich2006}. We also refer the reader to \cite{GO16} for some {\it directional counterparts} of the (unconstrained) Lipschitzian and related properties.\vspace*{0.03in}

Next we show that the relative Lipschitz-like property for constrained multifunctions is {\it equivalent} to the newly introduced {\it relative metric regularity} and {\it relative linear openness} properties of {\it inverse} mappings.
Prior to this, we present the following lemma whose proof is similar to \cite[Lemma~9.39]{Rockafellar1998} in $\mathbb{R}^n$ and thus is omitted.\vspace*{-0.03in}

\begin{lemma}\label{extended}
Let  $S:X\rightrightarrows Y$ be a set-valued mapping between Banach spaces, and let $\Omega\subset X$ and $\Theta \subset Y$ be nonempty sets. Given a pair $(\bar x,\bar y)\in\operatorname{gph}S|_\Omega^\Theta$, assume that $\operatorname{gph}S|_\Omega^\Theta$ is closed around this point. Then we have the equivalent assertions:

{\bf(i)} There exist $V\in\mathcal N(\bar x)$ and $W\in \mathcal N(\bar y)$ such that
$$
S|^\Theta(x')\cap W \subset S|^\Theta(x)+\kappa\|x-x'\|_XB_Y\;\mbox{ for all }\;x,x'\in\Omega\cap V.
$$

{\bf(ii)} There exist $V\in\mathcal N(\bar x)$ and $W\in \mathcal N(\bar y)$ such that
$$
S|^\Theta(x')\cap W\subset S|^\Theta (x)+\kappa\|x-x'\|_XB_Y\;\mbox{ for all } x\in\Omega\cap V\;\mbox{ and }\;x'\in \Omega.
$$
\end{lemma}\vspace*{-0.05in}

Now we are in a position to establish the aforementioned equivalences for the relative well-posedness properties (two of which are introduced in the theorem formulation) and provide {\it pointwise characterizations} of the new properties in terms of the relative mirror contingent coderivative from Definition~\ref{coderivatives}(iv).\vspace*{-0.03in}

\begin{theorem}\label{MRO} Let $S:X\rightrightarrows Y$ be a set-valued mapping between Banach spaces, and let $\Omega\subset X$ and $\Theta\subset Y$ be closed. Pick $(\bar x,\bar y)\in\gph S|_\Omega^\Theta$ and suppose that the set $\operatorname{gph}S|_\Omega^\Theta$ is closed around $(\bar x,\bar y)$. Then the following properties are equivalent:

{\bf(a)} The inverse $S^{-1}$ is Lipschitz-like relative to $\Theta\times\Omega$ around $(\bar y,\bar x)$.

{\bf(b)} The mapping $S$ is {\sc metrically regular relative to} $\Omega\times\Theta$ around $(\bar x,\bar y)$, i.e., there exist $V\in\mathcal{N}(\bar x)$, $W\in\mathcal{N}(\bar y)$, and $\kappa\ge 0$ such that
$$
d\big(x',S^{-1}|^\Omega(y)\big)\le\kappa d\big(y,S|^\Theta(x')\big)\;\mbox{ for all }\;x'\in \Omega\cap V,~y\in \Theta\cap W.
$$

{\bf(c)} The mapping $S$ is {\sc linearly open relative to} $\Omega\times \Theta$ around $(\bar x,\bar y)$, i.e., there exist $V\in\mathcal{N}(\bar x)$, $W\in\mathcal{N}(\bar y)$, and $\kappa\ge 0$ such that
$$
\big(S|^\Theta(x')+\varepsilon B_Y\big)\cap \Theta\cap W\subset S|_\Omega(x'+\kappa\varepsilon B_X)\;\mbox{ for all }\;x'\in \Omega\cap V,~\varepsilon>0.
$$

{\bf(d)} If $X$ and $Y$ are reflexive, $\Omega$ and $\Theta$ are convex, the {\sc relative mirror contingent coderivative criterion} holds: $S$ is mirror PSNC relative to $\Omega\times\Theta$ at $(\bar x,\bar y)$ and
\begin{equation*}
\ker D_m^*S_\Omega^\Theta(\bar x\vbl\bar y)=\{0\}.
\end{equation*}
\end{theorem}
\begin{proof}
${\bf[(a)\Longrightarrow(b)]}$: Apply Lemma~\ref{extended} to $S^{-1}$ and interpret the property in (a) as referring to the existence of $V\in\mathcal{N}(\bar x)$, $W\in\mathcal{N}(\bar y)$, and $\kappa\geq0$ such that
$$
S^{-1}|^\Omega(y')\cap V\subset S^{-1}|^\Omega(y)+\kappa\|y-y'\|_YB_X\;\mbox{ for all }\;y\in \Theta\cap W,~y'\in \Theta.
$$
Picking any $x'\in \Omega\cap V$ and $y\in\Theta\cap W$, it is obvious that $S^{-1}|^\Omega(y)\not=\emptyset$. It is trivial when $S|^\Theta(x')=\emptyset$. We assume that $S|^\Theta(x')\not=\emptyset$, and hence
$$
x'\in S^{-1}|^\Omega(y')\cap V\subset S^{-1}|^\Omega(y)+\kappa\|y-y'\|_YB_X\;\mbox{ for all }\;y'\in S|^\Theta(x'),
$$
which can be rewritten as the estimate
$$
d\big(x',S^{-1}|^\Omega(y)\big)\leq\kappa\|y-y'\|_Y\;\mbox{ whenever }\; y\in \Theta\cap W,~y'\in S|^\Theta(x')
$$
implying (b) by the minimization of the right-hand side quantity over $y'\in S|^\Theta(x')$.

${\bf[(b)\Longrightarrow(c)]:}$~ Fix $x'\in \Omega\cap V$, $\varepsilon>0$ and assume without loss of generality that $\big(S|^\Theta(x')+\varepsilon B_Y\big)\cap \Theta\cap W\neq\emptyset$. For any $y\in\big(S|^\Theta(x')+\varepsilon B_Y\big)\cap \Theta\cap W$, we have $y\in \Theta \cap W$ and $d(y,S|^\Theta(x'))\leq\varepsilon$. Then it follows from (b) that $d(x',S^{-1}|^\Omega(y))\leq\kappa\varepsilon$, which implies in turn that $y\in S|_\Omega(x'+\kappa'\varepsilon B_X)$ for any $\kappa'>\kappa$ and thus verifies (c).

${\bf[(c)\Longrightarrow(a)]}$:~Fix $y\in \Theta\cap W$ and $y'\in \Theta$. We may assume that $S^{-1}|^\Omega(y')\cap V\ne\emptyset$. Then $x'\in \Omega\cap V$ for any $x'\in S^{-1}|^\Omega(y')\cap V$. Taking $\varepsilon:=\|y-y'\|_Y$ and using $\bf(c)$ together with $y\in \Theta\cap W$ tell us that
\begin{eqnarray*}
&& y\in \big(y' +\varepsilon B_Y\big)\cap \Theta\cap W \subset \big(S(x') \cap \Theta + \varepsilon B_Y \big)\cap \Theta\cap W \\
&& \hskip0.25cm = \big(S|^\Theta(x')+\varepsilon B_Y\big)\cap \Theta\cap W\subset S|_\Omega(x'+\kappa\varepsilon B_X),
\end{eqnarray*}
which yields the relationships
$$
x'\in S^{-1}|^\Omega(y)+\kappa\varepsilon B_X=S^{-1}|^\Omega(y)+\kappa\|y-y'\|_YB_X
$$
and therefore shows, by employing Lemma \ref{extended} to $S^{-1}$, that assertion (a) is satisfied.

{\bf[(a)$\Longleftrightarrow$(d)]}:~We deduce from Lemma \ref{mirror}(ii) that
$$
D_M^*S_\Theta^{-1}(\bar y\vbl\bar x)(0)=\{0\}\Longleftrightarrow\ker D_m^*S^\Theta(\bar x\vbl\bar y)=0.
$$
Using finally Theorem \ref{criterion} and Lemma \ref{mirror}(i) verifies the equivalence between (a) and (d) and thus completes the proof of all the claimed assertions.
\end{proof}

Next we present a consequence of Theorem~\ref{MRO}  while combining it with the classical open mapping theorem in general Banach spaces.\vspace*{-0.03in}

\begin{corollary}\label{linear}
Let $A\in\mathcal{L}(X,Y)$, where both spaces $X$ and $Y$ are Banach. Suppose that the subspace $Y_1:=AX$ is closed in $Y$ and that $(\bar y,\bar x)\in\operatorname{gph} A^{-1}|_{Y_1}$. Then the mapping $A^{-1}$ is Lipschitz-like  relative to $Y_1 \times X$ around $(\bar y,\bar x)$.
\end{corollary}
\begin{proof}\vspace*{-0.02in}
Since $Y_1=AX$ is closed, it is a Banach space. Then the linear mapping $A:X\to Y_1$ is surjective. The classical open mapping theorem tells us that there exists $\kappa>0$ such that $B_{Y_1}\subset\kappa A(B_X)$. Thus it holds that
$$
Ax+\varepsilon B_{Y_1}\subset Ax+\kappa\varepsilon A(B_X)=A(x+\kappa\varepsilon B_X)\;\mbox{ whenever }\;x\in X~\operatorname{and}~\varepsilon>0,
$$
which implies that for all $x\in X$ and $\varepsilon>0$ we have the inclusion
$$
(Ax+\varepsilon B_Y)\cap Y_1\subset A(x+\kappa\varepsilon B_X)
$$
due to $AX=Y_1$. Therefore, $A$ is linearly open relative to $X \times Y_1$ around $(\bar x,\bar y)$. Using the equivalence between (a) and (c) in Theorem~\ref{MRO} verifies that $A^{-1}$ has the Lipschitz-like property relative to $Y_1 \times X$ around $(\bar y,\bar x)$, which is claimed.
\end{proof}

Yet another issue of its own interest happens to be of high importance for our analysis in infinite-dimensional spaces. This addresses the interplay between the weak and strong convergence of sequences. The classical {\it Kadec-Klee property} says that
\begin{equation*}
\big[x_k\rightharpoonup x,~ \|x_k\|_X\to \|x\|_X\big]\Longleftrightarrow x_k\to x~~\text{as}~k\to\infty.
\end{equation*}
As well known (see, e.g., \cite[Proposition~1.4, Chapter II]{deville}), every locally uniformly convex Banach space with a strictly convex norm enjoys the Kadec-Klee property. \vspace*{0.03in}

By applying Corollary \ref{linear}, we now derive an important convergence result of this type valid in {\it any Banach space}.  The authors are grateful to one reviewer for providing a proof of this result by constructing an isomorphism between $X\diagup_{\mathrm{ker}A}:=\{\hat{x}:=x+\mathrm{ker}A\ |\;x\in X\}$ and $AX$ (assumed to be closed) and applying the assumption that $\mathrm{ker}A$ is finite-dimensional. Nevertheless, to demonstrate the application of the relative Lipschitz-like property, we prefer to keep our (rather simpler) proof as follows.\vspace*{-0.02in}

\begin{corollary}\label{finiteker}
Let $X$ and $Y$ be Banach spaces, and let $A\in \mathcal{L}(X,Y)$. Suppose that $AX$ is closed in $Y$ and  that $\operatorname{dim}(\operatorname{ker} A)<\infty$. Then for any $x\in X$, we have
\begin{equation}\label{finiteker0}
\big[x_k\rightharpoonup x,~~A x_k\to A x\big]\Longrightarrow x_k\to x\;\mbox{ as }\; k\to\infty.
\end{equation}
\end{corollary}
\begin{proof}
Pick $x\in X$ and a sequence $x_k\rightharpoonup x$ such that { $Ax_k\to Ax$}. Denote $Y_1:=AX$ and deduce from Corollary~\ref{linear} that $A^{-1}$ has the Lipschitz-like property relative to $Y_1 \times X$ around $(A x,x)$ with a constant $\kappa\geq0$. For any $\varepsilon>0$, define
\begin{equation}\label{finiteker2}
\hat x_k:=x+\varepsilon(x_k-x)~~~\operatorname{and}~~~\hat y_k:=A \hat x_k.
\end{equation}
From the boundedness of $\{x_k\}$ and the relative Lipschitz-like property of $A^{-1}$, it follows for all $\varepsilon>0$ sufficiently small that
$$
\hat x_k\in A^{-1}(A x)+\kappa\|\hat y_k-A x\|_Y B_X=x+\operatorname{ker} A+\kappa\|\hat y_k-A x\|_YB_X.
$$
This gives us $\hat x_{1k}\in\operatorname{ker}A$ and $\hat x_{2k}\in x+\kappa\|\hat y_k-A x\|_YB_X$ such that
\begin{equation}\label{finiteker3}
\hat x_k=\hat x_{1k}+\hat x_{2k},\quad k\in\mathbb N.
\end{equation}
Using \eqref{finiteker2} and the assumed convergence on the left-hand side of \eqref{finiteker0} yields
$\hat x_k\rightharpoonup x$ and $\hat x_{2k}\to x$. Then we deduce from
\eqref{finiteker3} that $\hat x_{1k}\rightharpoonup 0$. Moreover, $\hat x_{1k}\in\operatorname{ker} A$ and $\operatorname{dim}(\operatorname{ker} A)<\infty$ ensure the norm convergence $\hat x_{1k}\to 0$. Combining the latter with \eqref{finiteker2} and \eqref{finiteker3} yields $x_k\to x$ as $k\to\infty$ and thus completes the proof.
\end{proof}

To study the relative Lipschitz-like property of multifunctions, the notion of {\it projectional coderivative} was introduced in \cite{Yang2021} in finite-dimensional spaces. Its {\it infinite-dimensional} version can be defined as follows. Let $S:X\rightrightarrows Y$ be a multifunction between two Banach spaces, and let $\Omega$ be a subset of $X$. Fixing $(\bar x,\bar y)\in\operatorname{gph} S|_\Omega$, we say that $x^*\in D_\Omega^{\mbox{\footnotesize proj}*}S(\bar x\vbl\bar y)(y^*)$ if there exist sequences $(x_k,y_k)\stackrel{\operatorname{gph} S|_\Omega}{\longrightarrow}(\bar x,\bar y)$  as $k\to\infty$ and $
(x_k^*,-y_k^*)\in N{ \big(}(x_k,y_k);\operatorname{gph} S|_\Omega{ \big)}$ for all $k$ such that ${ \mbox{proj}_{J_X(T(x_k;\Omega))}(x_k^*)}\rightharpoonup^* x^*$ and $y_k^*\to y^*$ as $k\to\infty$.  The following example demonstrates that this version of the projectional coderivative {\it does not} allow us to characterize the relative Lipschitz-like property in infinite-dimensional spaces as $D_\Omega^{\mbox{\footnotesize proj}*}S(\bar x\vbl\bar y)(0) \not= \{0\}$.

{ 
\begin{example} {\rm Let $X=L^{4/3}(0,1)$. It is easy to verify that
$$
J_{L^{4/3}(0,1)}(x(t))=\left\{
\begin{aligned}
&\displaystyle\|x\|_X^{2/3}|x(t)|^{-2/3}x(t)~~~&\mbox{if }~x(t)\not=0,\\
&0~~~&\mbox{if }~x(t)=0
\end{aligned}
\right.
$$
for all $x\in X$. Consider $S$ as $id_X$ and define the unit element $e \in X$ by $e(t):=1,\; t \in (0,1)$. Let $\Omega:=\{\mu e~|~\mu\geq0\}$ and $\Theta:=X$. It is obvious that $S$ is Lipschitz-like relative to $\Omega \times X$ around $(0,0)$. Let us now show that $e \in D_\Omega^{\mbox{\footnotesize proj}*}S(0|0)(0)$, which demonstrates that the projectional coderivative does not allow us to characterize the relative Lipschitz-like property. To proceed, take any $x_k\in\Omega\setminus\{0\} \to 0$ and $y_k = x_k$. We have that $T(x_k;\Omega)= \{\mu e~|~\mu\in\mathbb R\}$ and observe that $J_{L^{4/3}(0,1)}(e)=e$, which tells us  that 
$J_{L^{4/3}(0,1)}\big(T(x_k;\Omega)\big)=\{\mu e~|~\mu\in\mathbb R\}$. Let $\bar\tau\in\mathbb R$ be a unique (positive) root of the equation $2\tau^3-6\tau^2-1=0$. Define
$$
x_k^*(t):=\left\{
\begin{aligned}
&2\bar\tau~~~&\mbox{if}~t\in\left(0,\frac{1}{3}\right],\\
&-\bar\tau~~~&\mbox{if}~t\in\left(\frac{1}{3},1\right),
\end{aligned}
\right.
~~~\mbox{and}~~~y_k^*(t)\equiv 0,\quad k\in\mathbb N.
$$
Taking into account that the set $\gph S|_\Omega=\{(x,x)~|~x\in\Omega\}$ is convex and that
$$
\langle (x_k^*,y_k^*),(x',x')-(x,x)\rangle=\int_0^1x_k^*(t)(x'(t)-x(t))dt=0\;\mbox{ for all }\;x'\in\Omega,
$$
we deduce that $(x_k^*,-y_k^*)\in N((x_k,y_k);\gph S|_\Omega))$. It is easy to see that $\bar \mu=1$ solves the unconstrained optimization problem
$$
\min_{\mu\in\mathbb R} f_k(\mu):= \|x_k^*-\mu e\|_{X^*}^4=\int_0^1(x_k^*(t)-\mu)^4dt\;\mbox{ for each }\;k\in\mathbb N.
$$
Thus $e$ is the projection of $x_k^*$ to $J_{L^{4/3}(0,1)}(T(x_k;\Omega))$, and so $e\in D_\Omega^{\mbox{\footnotesize proj}*}S(0|0)(0)$.}
\end{example}}\vspace*{-0.05in}

\section{Variational Calculus}\label{sec:calculus} To be useful in applications, any construction of variational analysis must satisfy adequate {\it calculus rules}. This section addresses deriving the fundamental {\it chain rules} and {\it sum rules} for newly introduced mixed and normal relative contingent coderivatives as well as for the relative PSNC property. These calculus rules facilitate applications of the relative contingent coderivatives in verifying stability properties of some structured multifunctions. On the other hand, the established characterizations of the relative well-posedness and PSNC properties allow us to {\it efficiently verify} the calculus rules below for large classes of multifunctions.\vspace*{0.03in}

To proceed further, we introduce the following notions of inner semicontinuity and inner semicompactness of a multifunction relative to a set. These notions are extensions of the ones from \cite{Mordukhovich2006} to the case of constrained multifunctions.\vspace*{-0.03in}

\begin{definition}\label{inner-semi}
Let $S:X\rightrightarrows Y$ be a multifunction between Banach spaces, $\Omega\subset X$ be a nonempty set, and $(\bar x,\bar y) \in \operatorname{gph}S|_\Omega$. Then we say that:
\begin{description}
\item[(i)] $S$ is {\sc inner semicontinuous relative to} $\Omega$ at $(\bar x,\bar y)$, if for every sequence $x_k\stackrel{\operatorname{dom}S\cap \Omega}{\longrightarrow}\bar x$, there are $y_k\in S(x_k)$ 
converging to $\bar y$.
\item[(ii)] $S$ is {\sc inner semicompact relative to} $\Omega$ at $\bar x$, if for every sequence $x_k\stackrel{\operatorname{dom} S\cap\Omega}{\longrightarrow}\bar x$, there is a sequence   $y_k\in S(x_k)$ that contains a convergent   subsequence. \color{black}
\end{description}
\end{definition}\vspace*{0.03in}

Let us start with deriving {\it chain rules} for the  {\it relative contingent coderivatives} from Definition~\ref{coderivatives}(ii,iii). For brevity, we restrict ourselves to the case of {\it inner semicontinuity} of the auxiliary mapping below with discussing the case of {\it inner semicompactness} after the proof of the theorem, where $X$ and $\Omega$ are used as $X \times Z$ and $\Omega \times \Theta$ repsectively. For the notation simplicity, $D_M^*S_2^\Theta(\bar y\vbl\bar z)$ and $D_m^*S_{1\Omega}(\bar x\vbl\bar y)$ are the short forms of $D_M^*(S_2)^\Theta(\bar y\vbl\bar z)$ and $D_m^*(S_1)_\Omega(\bar x\vbl\bar y)$, respectively; the same are applied when subscripts $M$ and $m$ are replaced by $N$. \vspace*{-0.03in}

\begin{theorem}\label{chain-rule}
Let $S_1:X\rightrightarrows Y$ and $S_2:Y\rightrightarrows Z$ be multifunctions between reflexive Banach spaces, and let $\Omega\subset X$ and $\Theta\subset Z$ be closed and convex. Consider the composition $S:=S_2\circ S_1\colon X\rightrightarrows Z$ and define the multifunction $G\colon X\times Z\rightrightarrows Y$ by
\begin{equation}\label{G-semi}
G(x,z):=S_1(x)\cap S_2^{-1}(z)=\big\{y\in S_1(x)~\big|~z\in S_2(y)\big\}\;\mbox{ for }\;x\in X,\;z \in Z.
\end{equation}
Fixing $(\bar x,\bar z)\in\operatorname{gph}S|_\Omega^\Theta$ and $\bar y\in G(\bar x,\bar z)$, suppose that $\operatorname{gph}S_1|_\Omega$ and $\operatorname{gph}S_2|^\Theta$ are closed around $(\bar x,\bar y)$ and
$(\bar y,\bar z)$, respectively, and that $G$ is inner semicontinuous relative to $\Omega\times \Theta$ at $(\bar x,\bar z,\bar y)$. The following constrained coderivative chain rules hold:

{\bf(i)} Assume that either $S_1$ is mirror PSNC relative to $\Omega\times Y$ at $(\bar x,\bar y)$ or $S_2$ is PSNC relative to $Y\times\Theta$ at $(\bar y,\bar z)$, and that the mixed qualification condition
\begin{equation}\label{mixedQC1}
\ker D_m^*S_{1\Omega}(\bar x\vbl\bar y)\cap D_M^*S_2^\Theta(\bar y\vbl\bar z)(0)=\{0\}
\end{equation}
is satisfied. Then for all $z^*\in Z^*$, we have the inclusions
\begin{eqnarray}
D_M^*S_\Omega^\Theta(\bar x\vbl\bar z)(z^*)&\subset& D_N^* S_{1\Omega} (\bar x\vbl\bar y)\circ D_M^* S_2^\Theta(\bar y\vbl\bar z)(z^*),\label{chain-rule1}\\
D_N^*S_\Omega^\Theta(\bar x\vbl\bar z)(z^*)&\subset& D_N^* S_{1\Omega} (\bar x\vbl\bar y)\circ D_N^* S_2^\Theta(\bar y\vbl\bar z)(z^*)\label{chain-rule2}
\end{eqnarray}
while noting that the subscript $N$ in {\rm(\ref{chain-rule1})} is not a typo.

{\bf(ii)} If $\Theta=Z$ and $S_2:=f\colon Y\to Z$ is single-valued and strictly differentiable at $\bar y$, then we have the enhanced inclusion for the relative mixed contingent coderivative:
\begin{equation}\label{chain-rule3}
D_M^*S_\Omega(\bar x\vbl\bar z)(z^*)\subset D_M^* S_{1\Omega} (\bar x\vbl\bar y)\big(\nabla f(\bar y)^*z^*\big) \;\mbox{ whenever }\;z^*\in Z^*.
\end{equation}
\end{theorem}

\begin{proof}
To justify assertion (i), we just verify \eqref{chain-rule1} while observing that the proof of \eqref{chain-rule2} is similar and can be omitted. To proceed with \eqref{chain-rule1}, pick any element $x^*\in D_M^*S_\Omega^\Theta(\bar x\vbl\bar z)(z^*)$ and find sequences
$\varepsilon_k\downarrow0$, $(x_k,z_k)\stackrel{\operatorname{gph} S|_\Omega^\Theta}{\longrightarrow}(\bar x,\bar z)$, and
\begin{equation}\label{chain-rule4}
(x_k^*,-z_k^*)\in\widehat N_{\varepsilon_k}\big((x_k,z_k);\operatorname{gph} S|_\Omega^\Theta\big)\cap[J_X\big(T(x_k;\Omega)\big)\times J_Z\big(T(z_k;\Theta)\big)], \quad k\in\mathbb N,
\end{equation}
along which we have the convergences $z_k^*\to z^*$, $x_k^*\rightharpoonup^*x^*$ as $k\to\infty$.
Since $G$ is inner semicontinuous relative to $\Omega\times \Theta$ at $(\bar x,\bar z,\bar y)$, there are $y_k\in G(x_k,z_k)$ with $y_k\to \bar y$. Moreover, it follows from Lemma~\ref{JT} and \eqref{chain-rule4} that there exist $\delta_k>0$ such that
\begin{equation}\label{chain-rule5}
\begin{aligned}
x_k^*&\in J_X\big(T(x;\Omega)\big)+\varepsilon_k B_{X^*}\;\mbox{ for all }\;x\in \Omega\cap(x_k+\delta_k B_X)\;\mbox{ and }\;k\in\mathbb N,\\
- z_k^*&\in J_Z\big(T(z;\Theta)\big)+\varepsilon_k B_{Z^*}\;\mbox{ for all }\;z\in \Omega\cap(z_k+\delta_k B_Z)\;\mbox{ and }\;k\in\mathbb N.
\end{aligned}
\end{equation}
Let $C:=\{(x,y,z)\in X\times Y\times Z|y\in S_1|_\Omega(x),z\in S_2|^\Theta(y)\}$. It follows from \eqref{chain-rule4} that
$$
\begin{aligned}
&\limsup\limits_{(x,y,z)\stackrel{C}{\longrightarrow}(x_k,y_k,z_k)}
\frac{\langle(x_k^*,0,-z_k^*),(x,y,z)-(x_k,y_k,z_k)\rangle}{\|x-x_k\|_X+\|y-y_k\|_Y+\|z-z_k\|_Z}\\
\leq&\limsup\limits_{(x,z)\stackrel{\operatorname{gph} S|_\Omega^\Theta}{\longrightarrow}(x_k,z_k)}
\frac{\langle(x_k^*,-z_k^*),(x,z)-(x_k,z_k)\rangle^+}{\|x-x_k\|_X+\|z-z_k\|_Z}\le~\varepsilon_k,
\end{aligned}
$$
which readily verifies the inclusions
\begin{equation}\label{chain-rule6}
(x_k^*,0,-z_k^*)\in\widehat N_{\varepsilon_k}\big((x_k,y_k,z_k);C\big)\;\mbox{ for all }\; k\in\mathbb N.
\end{equation}
Defining $C_1:=\gph S_1|_\Omega\times Z$ and $C_2:=X\times\gph S_2|^\Theta$ gives us $C=C_1\cap C_2$. Let $\xi_k:=\min\{\delta_k,\varepsilon_k\}$, $w_k:=(x_k,y_k,z_k)$, and $w_k^*:=(x_k^*,0,-z_k^*)$. Endow the space { $W:=X\times Y\times Z$}  with the norm $\|(x,y,z)\|_W:=\sqrt{\|x\|_X^2+\|y\|_Y^2+\|z\|_Z^2}$. It follows from \eqref{chain-rule6} and Proposition~\ref{fuzzy} that there exist sequences
\begin{eqnarray}
&&\hat w_k\in C_1\cap(w_k+\xi_k B_W),~~\tilde w_k\in C_2\cap(w_k+\xi_k B_W),\label{chain-rule7}\\
&&\lambda_k\geq0,~~\hat w_k^*\in\widehat N(\hat w_k;C_1),~~
\tilde w_k^*\in\widehat N(\tilde w_k;C_2) \nonumber
\end{eqnarray}
such that for each $k\in\mathbb N$ we have the inequalities
\begin{equation}\label{chain-rule9}
\|\lambda_k w_k^*-\hat w_k^*-\tilde w_k^*\|_{W^*}\leq4\varepsilon_k,\quad
1-2\varepsilon_k\leq\max\big\{\lambda_k,\|\hat w_k^*\|_{W^*}\big\}\leq1+2\varepsilon_k.
\end{equation}
Denote $\hat w_k:=(\hat x_k,\hat y_k,\hat z_k)$, $\tilde w_k:=(\tilde x_k,\tilde y_k,\tilde z_k)$ and conclude from \eqref{chain-rule7} that
\begin{equation}\label{chain-rule10}
(\hat x_k,\hat y_k)\stackrel{\operatorname{gph} S_1|_\Omega}{\longrightarrow}(\bar x,\bar y)~~~\mbox{and}~~~(\tilde y_k,\tilde z_k)\stackrel{\operatorname{gph} S_2|^\Theta}{\longrightarrow}(\bar y,\bar z).
\end{equation}
Moreover, the structures of $C_1$ and $C_2$ lead us to the representations
\begin{equation}\label{chain-rule11}
\hat w_k^*=(\hat x_k^*,-\hat y_k^*,0)~~~\mbox{and}~~~\tilde w_k^*=(0,\tilde y_k^*,-\tilde z_k^*)
\end{equation}
together with the inclusions
\begin{equation}\label{chain-rule12}
(\hat x_k^*,-\hat y_k^*)\in\widehat N\big((\hat x_k,\hat y_k);\gph S_1|_\Omega\big)~~~\mbox{and}~~~(\tilde y_k^*,-\tilde z_k^*)\in\widehat N\big((\tilde y_k,\tilde z_k);\gph S_2|^\Theta\big).
\end{equation}
We now claim that there exists $\lambda_0>0$ such that $\lambda_k\geq\lambda_0$ for all $k\in\mathbb N$. Indeed, supposing the contrary gives us without loss of generality that $\lambda_k\downarrow0$ as $k\to\infty$. Then it follows from \eqref{chain-rule9},  \eqref{chain-rule11}, { and the boundedness of
$\{w^*_k\}$} that
\begin{equation}\label{chain-rule13}
\hat x_k^*\to0,~~\hat y_k^*-\tilde y_k^*\to0,~~\mbox{and}~~\tilde z_k^*\to 0\;\mbox{ as }\;k\to\infty.
\end{equation}
Moreover, \eqref{chain-rule9} tells us that $\{\hat y_k^*\}$ is bounded, and so there exists $y^*\in Y^*$ with
\begin{equation}\label{chain-rule14}
\|\hat y_k^*\|_{Y^*}\to1~~~\mbox{and}~~~\hat y_k^*\rightharpoonup^* y^*\;\mbox{ as }\;k\to\infty.
\end{equation}
It clearly follows from \eqref{chain-rule12} that
$$
\begin{aligned}
&(0,-\hat y_k^*)\in\widehat N_{\|\hat x_k^*\|_{X^*}}\big((\hat x_k,\hat y_k);\gph S_1|_\Omega\big)\cap\big[{ J_X\big(T(\hat x_k;\Omega)\big)}\times Y^*\big]\\
&(\tilde y_k^*,0)\in\widehat N_{\|\tilde z_k^*\|_{Z^*}}\big((\tilde y_k,\tilde z_k);\gph S_2|^\Theta\big)\cap\big[Y^*\times J_Z\big(T(\hat z_k;\Theta)\big)\big].
\end{aligned}
$$
Combining the latter with \eqref{chain-rule10}, \eqref{chain-rule13} and \eqref{chain-rule14} yields
$0\in D_m^*S_{1\Omega}(\bar x\vbl\bar y)(y^*)$ and $y^*\in D_M^*S_2^\Theta(\bar y\vbl\bar z)(0)$,

and hence $y^*\in \ker D_m^*S_{1\Omega}(\bar x\vbl\bar y)\cap D_M^*S_2^\Theta(\bar y\vbl\bar z)(0)$. Then it follows from \eqref{mixedQC1} that $y^*=0$. The assumptions that either $S_2$ is PSNC relative to $Y\times\Theta$ at $(\bar y,\bar z)$, or $S_1$ is mirror PSNC relative to $\Omega\times Y$ at $(\bar x,\bar y)$ ensure that $\hat y_k^*\to 0$, which contradicts \eqref{chain-rule14} and thus justifies the claimed existence of $\lambda_0>0$.

Therefore, we assume without loss of generality that $\lambda_k=1$ for all $k\in\mathbb N$ and then deduce from \eqref{chain-rule9} that
\begin{equation*}\label{chain-rule15}
\hat y_k^*-\tilde y_k^*\to0,~~~\|\tilde z_k^*-z_k^*\|\leq4\varepsilon_k,~~~\mbox{and}~~~\|x_k^*-\hat x_k^*\|_{X^*}\leq4\varepsilon_k~~\mbox{for~all}~k\in\mathbb N.
\end{equation*}
Using \eqref{chain-rule9} again, we get that $\{\hat y_k^*\}$ is bounded and there exists $y^*\in Y^*$ such that
\begin{equation}\label{chain-rule16}
\hat y_k^*\rightharpoonup^* y^*~~~\mbox{and}~~~\tilde y_k^*\rightharpoonup^* y^*\;\mbox{ as }\;k\to\infty.
\end{equation}
Furthermore, it follows from \eqref{chain-rule5} that there exist $\bar x_k^*\in J_X\big(T(\hat x_k;\Omega)\big)$ with $\|\bar x_k^*-x_k^*\|_{X^*}\leq\varepsilon_k$ and $\bar z_k^*\in J_Z\big(T(\tilde z_k;\Theta)\big)$ with $\|\bar z_k^*-z_k^*\|_{Z^*}\leq\varepsilon_k$ for each $k\in\mathbb N$. Hence $\bar x_k^*\rightharpoonup^*x^*$, $\bar z_k^*\to z^*$ as $k\to\infty$, and
\begin{eqnarray}
&&(\bar x_k^*,-\hat y_k^*)\in\widehat N_{5\varepsilon_k}\big((\hat x_k,\hat y_k);\gph S_1|_\Omega\big)\cap\big[J_X\big(T(\hat x_k;\Omega)\big)\times Y^*\big],\label{chain-rule17}\\
&&(\tilde y_k^*,-\bar z_k^*)\in\widehat N_{5\varepsilon_k}\big((\tilde y_k,\tilde z_k);\gph S_2|^\Theta\big)\cap\big[Y^*\times J_Z\big(T(\tilde z_k;\Theta)\big)\big],\quad k\in\mathbb N. \nonumber
\end{eqnarray}
The above inclusions, together with \eqref{chain-rule10} and \eqref{chain-rule16}, yield $x^*\in D_N^*S_\Omega(\bar x\vbl\bar y)(y^*)$ and $y^*\in D_M^*S_2^\Theta(\bar x\vbl\bar y)(z^*)$ and thus justify  \eqref{chain-rule1}.

Let us verify (ii). Observe that the strict differentiability of $S_2=f$ ensures the fulfillment of the qualification condition \eqref{mixedQC1} and the relative PSNC property of $S_2$. 
Thus the proof of (i) still holds when replacing $S_2$ by $f$. By \eqref{chain-rule12} and \cite[Theorem 1.38]{Mordukhovich2006}, we get that $\tilde y_k^*=\nabla f(\bar x)^*\tilde z_k^*.$ Using $\|\tilde z_k^*- z_k^* \|_{Z^*} \leq 4 \varepsilon_k$ and $z^*_k \to z^*$ yields $\tilde z_k^*\to z^*$ and hence $\tilde y_k^*\to \nabla f(\bar x)^* z^*$ as $k\to\infty$. Together with $\tilde y_k^*\rightharpoonup^* y^*$ this gives us $\nabla f(\bar x)^* z^*= y^*$, which being combined with \eqref{chain-rule13} gives us $\hat y_k^*\to y^*$. Passing to the limit in \eqref{chain-rule17}, we arrive at  $x^*\in D_M^*(S_1)_\Omega(\bar x\vbl\bar y)(y^*)$ and hence justify \eqref{chain-rule3}. 
\end{proof}

In Theorem \ref{chain-rule}(ii), inclusion \eqref{chain-rule3} may not hold if $\Theta\ne Z$ as shown below.

\begin{example}
{\rm Let $X=\mathbb R, Y=Z=\ell^2$, $\Omega:=[-\frac{1}{2},\frac{1}{2}]$ and $\Theta:=\{z\in\ell^2~|~z_i=0,~i\geq2\}$. Define $f(y):=y$ for all $y\in Y$ and define the mapping $S_1$ by
$$
S_1(x):=\left\{
\begin{aligned}
&2^{-k}\mathbf e_k~~&&\mbox{if}~|x|=2^{-k}\;\mbox{for some }\;k \in \mathbb N\\
&2^{-1}\mathbf e_1~~&&\mbox{if}~|x|>2^{-1},\\
&0~~&&\mbox{if}~x=0,\\
&\mbox{linear}~~&&\mbox{otherwise},
\end{aligned}
\right.
$$
where $\mathbf e_k$ is given by $(\mathbf e_k)_i=1$ if $i=k$ and $(\mathbf e_k)_i=0$ if $i\ne k$. Take $\bar x=0$ and $\bar y=\bar z=0$. As $\bar x\in\inte \Omega$, it follows from \cite[Example~1.35] {Mordukhovich2006}  that
$
D_M^*S_{1\Omega}(\bar x\vbl\bar y)(y^*)=0 $ for all $y^*\in Y$.
We obviously see that the corresponding mapping $G$ defined in \eqref{G-semi} is inner semicontinuous relative to $\Omega\times\Theta$ and that
$
D_M^*(f\circ S_1)_{\Omega}^\Theta(\bar x|\bar z)(z^*)=\mathbb R$ for all $z^*\in\Theta$.
Therefore,  \eqref{chain-rule3} does not hold due to the relationships
$$
D_M^*(f\circ S_1)_{\Omega}^\Theta(\bar x\vbl\bar z)(z^*)=\mathbb R\not\subset\{0\}=D_M^* S_{1\Omega} (\bar x\vbl\bar y)\circ D_M^* f^\Theta(\bar y\vbl\bar z)(z^*)\;\mbox{ as }\;z^*\in\Theta,
$$
which confirms the claim of this example.}
\end{example}\vspace*{0.03in}

The second inclusion in \eqref{cod-normality} shows that under the strict differentiability of $f$, the chain rule \eqref{chain-rule3} is tighter than \eqref{chain-rule1}. Observe that assertions (i) and (ii) of Theorem~\ref{chain-rule} are independent from each other. Indeed, inclusion \eqref{chain-rule3} for the mixed contingent coderivative of the composition in (ii) is expressed via the {\it mixed} one $D_M^*S_\Omega^\Theta$, while not via its {\it normal} counterpart as in \eqref{chain-rule1}. Furthermore, it follows from the proof of Theorem~\ref{chain-rule}(i) that changing the inner semicontinuity assumption on mapping \eqref{G-semi} to its (weaker) {\it semicompactness} relative to $\Omega$ at $(\bar x,\bar z)$ leads us to the chain rule inclusions of types \eqref{chain-rule1} and \eqref{chain-rule2} with the replacement of $\bar y$ on the right-hand sides by the {\it union} over the set $G(\bar x,\bar z)$.\vspace*{0.02in}

Using now the obtained {\it coderivative characterizations} of {\it well-posedness}, we arrive at the efficient conditions for the fulfillment of the chain rules \eqref{chain-rule1} and \eqref{chain-rule2}.\vspace*{-0.03in}

{ 
\begin{corollary} In the general setting of Theorem~{\rm\ref{chain-rule}}, if either $S_1$ is metrically regular/linearly open relative to $\Omega\times Y$ around $(\bar x,\bar y)$ or
$S_2$ is Lipschitz-like relative to $Y\times\Theta$ around $(\bar y,\bar z)$, then both  chain rules in \eqref{chain-rule1} and \eqref{chain-rule2} are fulfilled.
\end{corollary}
\begin{proof}
Under the imposed well-posedness assumptions, the fulfillment of the qualification condition \eqref{mixedQC1} and the corresponding relative PSNC properties of Theorem~\ref{chain-rule}(i) are consequences of Theorems~\ref{criterion} and \ref{MRO}. Therefore, all the conclusions of the corollary follow from Theorem~\ref{chain-rule}(i).
\end{proof}}\vspace*{-0.03in}

The procedure developed in the proof of Theorem~\ref{chain-rule}(i) allows us to verify the preservation of the relative PSNC property of compositions under a more restrictive qualification condition in comparison with  \eqref{mixedQC1}.\vspace*{-0.03in}

\begin{proposition}\label{chain-rule-PSNC}In the setting of Theorem~{\rm\ref{chain-rule}}, assume that $S_1$ is PSNC relative to $\Omega \times Y$ at $(\bar x,\bar y)$, $S_2$ is PSNC relative to $Y\times \Theta$ at $(\bar y,\bar z)$, and the qualification condition
\begin{equation}\label{mixedQC1'}
D_M^*S_2^\Theta(\bar y\vbl\bar z)(0)\cap\ker D_N^*S_{1\Omega}(\bar x\vbl\bar y)=\{0\}
\end{equation}
holds. Then the composition $S$ is PSNC relative to $\Omega \times \Theta$ at $(\bar x,\bar z)$.
\end{proposition}\vspace*{-0.04in}
\begin{proof} It is similar to that of Theorem~\ref{chain-rule}(i); here is the  outline. In the proof above,  we may assume that $x_k^*\rightharpoonup^*0$ and $z_k^*\to0$ as $k\to\infty$ and $\lambda_k=1$ for all $k\in\mathbb N$. Similar to the proof of Theorem ~\ref{chain-rule}(i), an element $y^*\in D_M^*S_2^\Theta(\bar y\vbl\bar z)(0)\cap\ker D_N^*S_{1\Omega}(\bar x\vbl\bar y)$ is constructed. Using the qualification condition \eqref{mixedQC1'} yields $y^*=0$. A sequence $\{\tilde y_k^*\}$ is obtained from \eqref{chain-rule11}. Then we get from the relative PSNC property of $S_2$ that $\tilde y_k^*\to0$, which ensures that $\hat y_k^*\to 0$ from \eqref{chain-rule13}. It follows from the PSNC property of $S_1$ relative to $\Omega$ at $(\bar x,\bar y)$ that $\bar x_k^*\to0$, and so $x_k^*\to 0$, which therefore justifies the claimed relative PSNC property of $S$ at $(\bar x,\bar z)$.
\end{proof}\vspace*{-0.03in}

The rest of this section is devoted to deriving {\it sum rules} for the relative contingent coderivatives. Note that the structures of the sum rules below are different from the chain rules obtained above. In contrast to Theorem~\ref{chain-rule}, the qualification condition and the PSNC properties are now formulated {\it entirely} via the {\it relative} constructions.

To proceed in this way, we first establish a novel variational result of its own fundamental importance --- namely, the {\it extremal principle} concerning the new notion of {\it relative set extremality} in the following sense; cf.\ \cite{Mordukhovich2006} in the unconstrained case. In this relative set extremality, a product space of $X$ and $Y$ is considered which is motivated by the consideration that the perturbation of an element of a tangent cone may not still be in the tangent cone. As such the perturbation in the relative set extremality is only allowed in the space $Y$ and so our new relative set extremality is stronger than the usual set extremality in \cite{Mordukhovich2006}.
\vspace*{0.03in}

Before proceeding, we make a note that with a closed and convex set $\Theta$, the chain rules (\ref{chain-rule1}) and (\ref{chain-rule2}) in Theorem \ref{chain-rule} are able to provide sufficient conditions for the relative Lipschitz-like property of a composite multifunction. However, for the sum rule establishing below to provide sufficient conditions for the relative Lipschitz-like property of the sum of two multifunctions, the set $\Theta$ needs to be the whole space. To illustrate this, consider $\Omega:=[0,\infty)$, $\Theta:=\mathbb R\times\{0\}$, $S_1(x):=\{(x,0),(\sqrt x,x)\}$, and $S_2(x)=\{(x,0),(\sqrt x,-x)\}$ for all $x\geq0$. It is obvious that $S_i|_\Omega^\Theta(x)=\{(x,0)\}$ for all $x\in\Omega$, $i=1,2$, and thus both $S_1$ and $S_2$ are Lipschitz-like relative to $\Omega\times \Theta$ around $(0,0,0)$. We also have $(S_1+S_2)|_\Omega^\Theta(x)=\{(2x,0),(2\sqrt{x},0)\}$ for all $x\in\Omega$. It is clear that the modified constraint qualification in
 (\ref{sumQC}) holds but $(S_1+S_2)|_\Omega^\Theta(x)$ is not Lipschitz-like relative to $\Omega\times \Theta$ around $(0,0,0)$ and that all the conditions in Theorem \ref{sumrule} below are satisfied. Therefore, to establish the relative sum rules for the constrained coderivatives, the set $\Theta$ is taken as the whole space $Y$.\vspace*{0.03in}

 Now we introduce the new notion of {\it relative local extremality} for systems of sets.\vspace*{-0.03in}

\begin{definition}
Let $X$ and $Y$ be Banach spaces, $\Lambda_1$ and $\Lambda_2$ be two nonempty subsets of the product space $X\times Y$, $\Omega$ be a nonempty subset of $X$, and $(\bar x,\bar y)\in\Lambda_1\cap\Lambda_2$. We say that $(\bar x,\bar y)$ is a {\sc local extremal point} of the set system $\{\Lambda_1,\Lambda_2\}$
{\sc relative} to $\Omega\times Y$ if there exist a neighborhood $U$ of $(\bar x,\bar y)$ and a sequence of $\{b_k\}$ such that $b_k\to0$ as $k\to\infty$, ${\rm proj}_X(\Lambda_1\cup\Lambda_2)\subset \Omega$, and
$$
\big(\Lambda_1+(0,b_k)\big)\cap\Lambda_2\cap U=\emptyset ~~~\operatorname{for~all~sufficiently~large~}\;k\in\mathbb N.
$$
\end{definition}\vspace*{-0.03in}

The relative local extremality of two sets at a common point means that they can be locally ``pushed apart" in the common directions of one space at the point in question by a small perturbation of even one of them. Moreover, we notice that the relative local extremality is stronger than the unconstrained one in \cite{Mordukhovich2006}. By developing the proof of \cite[Theorem 2.10]{Mordukhovich2006}, we obtain the relative extremal principle as follows.\vspace*{-0.03in}

\begin{lemma}\label{extremal} Let $X$ and $Y$ be reflexive Banach spaces, $\Lambda_1$ and $\Lambda_2$ be nonempty closed subsets of the product space $X\times Y$, $\Omega$ be a closed and convex subset of $X$, and $(\bar x,\bar y)$ be a local extremal point of the set system $\{\Lambda_1,\Lambda_2\}$ relative to $\Omega\times Y$ and $Y$. Then for any $\varepsilon>0$, there exist $(x_i,y_i)\in \Lambda_i\cap[(\bar x+\varepsilon B_X)\times(\bar y+\varepsilon B_Y)]$ and $(x_i^*,y_i^*)\in X^*\times Y^*$, $i=1,2$, satisfying the conditions
\begin{equation}\label{extr0}
\begin{aligned}
&(x_i^*,y_i^*)\in\big[\widehat N\big((x_i,y_i);\Lambda_i\big)+\varepsilon (B_{X^*}\times B_{Y^*})\big]\cap {\big[J_X(T(x_i;\Omega))\times Y^*\big]},\\
&(x_1^*,y_1^*)+(x_2^*,y_2^*)=0,~~~\|(x_1^*,y_1^*)\|_{X^*\times Y^*}+\|(x_2^*,y_2^*)\|_{X^*\times Y^*}=1.
\end{aligned}
\end{equation}
\end{lemma}
\begin{proof}
The assumed set extremality gives us a neighborhood $U$ of $(\bar x,\bar y)$ such that for any $\varepsilon>0$, there exist $b\in Y$ with $\|b\|_Y\leq\varepsilon^3/2$ and $(\Lambda_1+(0,b))\cap\Lambda_2\cap U=\emptyset$. Suppose for simplicity that $U=X\times Y$ and $\varepsilon<1/2$. Clearly, $(\bar x,\bar y)$ is a local extremal point of the system $\{\Lambda_1,\Lambda_2\}$ in the sense of \cite{Mordukhovich2006}. Moreover, note that the product space $X\times Y$ normed with $\|(x,y)\|_{X\times Y}:=\sqrt{\|x\|_X^2+\|y\|_Y^2}$ is a Fr$\acute{\text{e}}$chet smooth
space. Following the proof of \cite[Theorem 2.10]{Mordukhovich2006}, we find elements
$$
(\bar x_i,\bar y_i)\in\Lambda_i\cap[(\bar x+\varepsilon B_X)\times(\bar y+\varepsilon B_Y)] ~\mbox{and}~(u_i^*,v_i^*)\in\widehat N\big((\bar x_i,\bar y_i);\Lambda_i\big),~~i=1,2,
$$
satisfying the inclusions
$$
(u_1^*,v_1^*)\in-(x^*,y^*)+\varepsilon B_{X^*\times Y^*}~~\mbox{and}~~(u_2^*,v_2^*)\in(x^*,y^*)+\varepsilon B_{X^*\times Y^*},
$$
$$
(x^*,y^*)=\Big(\frac{J_X(\bar x_1-\bar x_2)}{\sqrt{\|\bar x_1-\bar x_2\|_X^2+\|\bar y_1-\bar y_2+b\|_Y^2}},\frac{J_Y(\bar y_1-\bar y_2+b)}{\sqrt{\|\bar x_1-\bar x_2\|_X^2+\|\bar y_1-\bar y_2+b\|_Y^2}}\Big),
$$
{ where the calculation of the derivative of the norm in the last equality comes from \eqref{diffduality}}. Since $\Omega$ is convex and $\bar x_i\in{\rm proj}_X(\Lambda_i)\subset\Omega, i=1,2,$ we get $\bar x_2-\bar x_1\in T(\bar x_1;\Omega)$ and $\bar x_1-\bar x_2\in T(\bar x_2;\Omega)$, which yields $-x^*\in J_X\big(T(\bar x_1;\Omega)\big)$ and $x^*\in J_X\big(T(\bar x_2;\Omega)\big)$. \color{black}  Putting $(x_i,y_i):=(\bar x_i,\bar y_i)$ and $(x_i^*,y_i^*):=(-1)^i(x^*,y^*)/2$ verifies \eqref{extr0}.
\end{proof}

From now on, we assume that the spaces under consideration are {\it Hilbert}. This is needed to adjust the {\it tangent} (primal) and {\it normal} (dual) elements in the our mixed contingent coderivatives. The next result, which follows from the relative extremal principle of Lemma~\ref{extremal}, establishes a new {\it relative fuzzy intersection rule} that extends \cite[Lemma~3.1]{Mordukhovich2006} from the unconstrained version. The unconstrained fuzzy rules of the latter type have been well recognized in variational analysis; see, e.g., \cite{Fabian2024,ioffe,Mordukhovich2006}.\vspace*{-0.03in}

\begin{lemma}\label{fuzzy-relative}
Let $X$ be a Hilbert space, and let $G_1,G_2\subset X$ be closed around $(\bar x,\bar y)\in G_1\cap G_2$. Suppose that $G_1\cup G_2\subset\Omega$ for some closed and convex set $\Omega\subset X$. Fix $\varepsilon>0$ and $x^*\in\widehat N_\varepsilon (\bar x;G_1\cap G_2)\cap T(\bar x;\Omega)$. Then for any $\gamma>0$, there exist $\lambda\geq0$,  $x_i\in G_i\cap[\bar x+\gamma B_X]$, and $x_i^*\in[\widehat N(x_i;G_i)+(\varepsilon+\gamma)B_X]\cap T(x_i;\Omega)$, $i=1,2$, such that
\begin{equation}\label{fuzzy-relative0}
\lambda x^*-x_1^*-x_2^*\in\gamma B_X,~~\max\{\lambda,\|x_1^*\|_X\}=1.
\end{equation}
\end{lemma}
\begin{proof}
Taking any $\gamma>0$ and $x^*\in\widehat N_\varepsilon (\bar x;G_1\cap G_2)\cap T(\bar x;\Omega)$, we deduce from Proposition~\ref{JT} that there exists $\rho>0$ such that
\begin{equation}\label{fuzzy-relative1}
x^*\in T(x;\Omega)+(\gamma/30)B_X\;\mbox{ for all }\; x\in\Omega\cap[\bar x+\rho B_X].
\end{equation}
Define $\tau:=\frac{1}{30}\min\{\gamma,\rho,1\}$ \color{black} and find a neighborhood $U$ of $\bar x$ with
\begin{equation}\label{fuzzy-relative2}
\langle x^*,x-\bar x\rangle-(\varepsilon+\tau)\|x-\bar x\|_X\le 0\;\mbox{ on }\;G_1\cap G_2\cap U.
\end{equation}
Consider the closed subsets of $X\times \mathbb R$ given by
\begin{align*}
&\Lambda_1:=\big\{(x,\alpha)\in X\times \mathbb R~\big|~x\in G_1,~\alpha\ge 0\big\},\\
&\Lambda_2:=\{(x,\alpha)\in X\times \mathbb R~|~x\in G_2,~\alpha\leq \langle x^*,x-\bar x\rangle-(\varepsilon+\tau)\|x-\bar x\|_X\}
\end{align*}
and observe that $(\bar x,0)\in \Lambda_1\cap \Lambda_2$ and that
$(\Lambda_1+(0,\nu))\cap\Lambda_2\cap(U\times \mathbb R)=\emptyset$ as $\nu>0$
due to \eqref{fuzzy-relative2} and the structure of $\Lambda_i$. Thus $(\bar x,0)$ is a local extremal point relative to $\Omega$ of $\{\Lambda_1,\Lambda_2\}$. Applying Lemma~\ref{extremal} to this system in the space $X\times \mathbb R$ with the norm $\|(x,\alpha)\|:=\sqrt{\|x\|_X^2+|\alpha|^2}$ {  and $\varepsilon:=\tau/6$ in 
Lemma~\ref{extremal}}, gives us $(\bar x_i,\alpha_i)\in\Lambda_i$ and
\begin{equation}\label{fuzzy-relative3}
(\hat x_i^*,\lambda_i)\in\widehat N\big((\bar x_i,\alpha_i);\Lambda_i\big)\cap[T(\bar x_i;\Omega)\times \mathbb R+\tau(B_X\times \mathbb R)]
\end{equation}
for $i=1,2$ satisfying the conditions
\begin{align}
&\|\bar x_i-\bar x\|_X+|\alpha_i|\leq\tau,\quad i=1,2,\notag\\
&\|\hat x_1^*+\hat x_2^*\|_X+|\lambda_1+\lambda_2|  \leq\tau,\label{fuzzy-relative5}\\
&{ 0.5-\tau\leq\sqrt{\|\hat x_i^*\|_X^2+|\lambda_i|^2}\leq 0.5+\tau},\quad i=1,2.\label{fuzzy-relative6}
\end{align}
It is easy to deduce from \eqref{fuzzy-relative6} the estimate
\begin{equation}\label{fuzzy1'}
\max\{\|u^*\|_X,|\lambda_i|\}\ge 1/5~~~\mbox{if }~u^*\in \hat x_i^*+2\tau B_X.
\end{equation}
{ 
Indeed, for any $u^*\in \hat x_i^*+2\tau B_X$, we get from $\tau\leq\frac{1}{30}$ \color{black} that
$$
\begin{aligned}
\max\{\|u^*\|_X,|\lambda_i|\}\geq &\max\{\|\hat x_i^*\|_X,|\lambda_i|\}-2\tau=\sqrt{\max\{\|\hat x_i^*\|_X^2,|\lambda_i|^2\}}-2\tau\\
\geq&\sqrt{\frac{\|\hat x_i^*\|_{X}^2+|\lambda_i|^2}{2}}-2\tau\geq\frac{0.5-\tau}{\sqrt2}-2\tau\geq\frac{1}{5}.
\end{aligned}
$$}
It follows from \eqref{fuzzy-relative3} and $\alpha\geq 0$ in the definition of $\Lambda_1$ that $\lambda_1\le 0$, $\hat x_1^*\in\widehat N(\bar x_1;G_1)\cap[T(\bar x_1;\Omega)+\tau B_X]$, which ensures the existence of  $u_1^*\in \hat x_1^*+\tau B_X$ such that
\begin{equation}\label{u1}
u_1^*\in\big[\widehat N(\bar x_1;G_1)+\tau B_X\big]\cap T(\bar x_1;\Omega).
\end{equation}
Moreover, \eqref{fuzzy-relative3} also tells us that
\begin{equation}\label{fuzzy4}
\limsup\limits_{(x,\alpha)\stackrel{\Lambda_2}{\longrightarrow}(\bar x_2,\alpha_2)}
\frac{\langle \hat x_2^*,x-\bar x_2\rangle+\lambda_2(\alpha-\alpha_2)}{\sqrt{\|x-\bar x_2\|_X^2+|\alpha-\alpha_2|^2}}\leq0.
\end{equation}
Exploiting the structure of the set $\Lambda_2$, we get $\lambda_2\geq0$ and
\begin{equation}\label{fuzzy5}
\alpha_2\leq\langle x^*,{ \bar x_2}-\bar x\rangle-(\varepsilon+\tau)\|{ \bar x_2}-\bar x\|_X.
\end{equation}
If the inequality in \eqref{fuzzy5} is strict, then \eqref{fuzzy4} yields $\lambda_2=0$ and so ${ \hat x_2^*}
\in\widehat N(\bar x_2;G_2)\cap[T(\bar x_2;\Omega)+\tau B_X]$. Hence $|\lambda_1|\leq 1/30$ \color{black} due to \eqref{fuzzy-relative5}, and there exists $u_2^*\in [\hat x_2^*+\tau B_X]\cap T(\bar x_2;\Omega)$ satisfying the inclusion
$$
u_2^*\in\big[\widehat N(\bar x_2;G_2)+\tau B_X]\cap T(\bar x_2;\Omega).
$$
In this case, \eqref{fuzzy1'} gives us the estimate $\|u_1^*\|_X\geq 1/5$, and therefore $\eqref{fuzzy-relative0}$ holds with $x_i=\bar x_i$, $x_i^*=u_i^*/\|u_1^*\|_X$ as $i=1,2$, and $\lambda=0$.

Examine now the case where \eqref{fuzzy5} becomes an equality. Take $(x,\alpha)\in\Lambda_2$ with
$$
\alpha=\langle x^*,x-\bar x\rangle-(\varepsilon+\tau)\|x-\bar x\|_X,~~(x,y)\in G_2\setminus\{\bar x_2\}.
$$
Substituting the latter into \eqref{fuzzy4} and denoting $\sigma:=\frac{\tau}{2+\varepsilon+\|x^*\|_X}$, we find a neighborhood $V$ of $x_2$ such that
\begin{equation}\label{fuzzy6}
\begin{array}{ll}
&\langle \hat x_2^*,x-\bar x_2\rangle+\lambda_2(\alpha-\alpha_2)
\leq\sigma\|(x,\alpha)-(\bar x_2,\alpha_2)\|_{X\times\mathbb R}\\
&\leq\sigma\big(\|x-\bar x_2\|_X+|\alpha-\alpha_2|\big)\;\mbox{ for all }\;x\in G_2\cap V
\end{array}
\end{equation}
with the corresponding number $\alpha$ satisfying the equality
$$
\alpha-\alpha_2=\langle x^*,x-\bar x_2\rangle+(\varepsilon+\tau)\big(\|\bar x_2-\bar x\|_X-\|x-\bar x\|_X\big).
$$
Then the triangle inequality tells us that
$$
|\alpha-\alpha_2|\leq\big(\|x^*\|_X+\varepsilon+\tau\big)\|x-\bar x_2 \color{black}\|_X.
$$
Observe further that the left-hand side $\vartheta$ in \eqref{fuzzy6} can be represented as
$$
\begin{aligned}
\vartheta=\langle \hat x_2^*+\lambda_2x^*,x-\bar x_2\rangle+\lambda_2(\varepsilon+\tau)\big(\|\bar x_2-\bar x\|_X-\|x-\bar x\|_X),\color{black}
\end{aligned}
$$
which allows us to deduce from \eqref{fuzzy-relative6} and \eqref{fuzzy6} the estimates
$$
\begin{aligned}
\langle \hat x_2^*+\lambda_2x^*,x-\bar x_2\rangle\leq&\big[\lambda_2(\varepsilon+\tau)+\sigma(\|x^*\|_X+\varepsilon+\tau+1)\big]\|x-\bar x_2\|_X\\
\leq&(\lambda_2\varepsilon+2\tau)\|x-\bar x_2\|_X
\end{aligned}
$$
for all $x\in G_2\cap V$. This readily yields
\begin{equation}\label{fuzzy7}
\hat x_2^*+\lambda_2x^*\in\widehat N_{\lambda_2\varepsilon+2\tau}(\bar x_2;G_2).
\end{equation}
It follows from $\hat x_2^*\in T(\bar x_2;\Omega)+\tau B_X$ and \color{black} Proposition~\ref{JT} that there is $\rho'\in(0,\tau)$ with
\begin{equation}\label{fuzzy8}
\hat x_2^*\in T(x;\Omega)+2\tau B_X\color{black}\;\mbox{ whenever }\;x\in \Omega\cap(\bar x_2+\rho' B_X).
\end{equation}
Then using the description of $\varepsilon$-normals given in \cite[p.\ 222]{Mordukhovich2006}) gives us $u\in G_2\cap[\bar x_2+\rho'B_X]$ satisfying the inclusion
\begin{equation}\label{ff}
\hat x_2^*+\lambda_2x^*\in\widehat N(u;G_2)+(\lambda_2\varepsilon+3\tau)B_X.
\end{equation}
Thus we get $x^*\in T(u;\Omega)+\frac{\gamma}{30}B_X$ and $\hat x_2^*\in T(u;\Omega)+2\tau B_X$ \color{black} from \eqref{fuzzy-relative1} and \eqref{fuzzy8}, respectively. This brings us to $\hat x_2^*+\lambda_2 x^*\in T(u;\Omega)+2\tau\color{black}+\frac{\lambda_2\gamma}{30})B_X$ and, being combined with \eqref{ff}, yields the existence of $u_2^*\in T(u;\Omega)$ such that
$$
u_2^*\in\widehat N(u;G_2)+\Big(\lambda_2\varepsilon+5\tau\color{black}+\frac{\lambda_2\gamma}{30}\Big)B_X.
$$
Furthermore, it follows from \eqref{fuzzy-relative5} that
$$
-\hat x_1^*=\hat x_2^*-(\hat x_1^*+\hat x_2^*)\in\hat x_2^*+\tau B_X,
$$
which implies that $-u_1^*\in\hat x_2^*+2\tau B_X$ by $u_1^*\in\hat x_1^*+\tau B_X$. Hence we get from \eqref{fuzzy1'} that $\eta:=\max\{\|u_1^*\|_X,\lambda_2\}\ge 1/5$. This finally ensures that $\eqref{fuzzy-relative0}$ holds with $\lambda:=\frac{\lambda_2}{\eta}$, $x_1:=\bar x_1$, $x_2:=u$, $x_1^*:=\frac{u_1^*}{\eta}$, and $x_2^*:=\frac{u_2^*}{\eta}$, which completes the proof.
\end{proof}

Now we are ready to establish sum rules for the relative contingent coderivatives. The proof of this theorem follows directly by replacing the fuzzy intersection rule in the demonstration of the sum rule in \cite[Theorem 3.10]{Mordukhovich2006} with our newly proposed constrained version given above. Consequently, the formal proof is omitted here.\vspace*{-0.03in}

\begin{theorem}\label{sumrule}
Let $X$ and $Y$ be Hilbert spaces, $\Omega\subset X$ be a closed and convex set, $S_i:X\rightrightarrows Y$ as $i=1,2$ with $(\bar x,\bar y)\in \operatorname{gph}(S_1+S_2)|_\Omega$. Define
\begin{equation}\label{G-sum}
G(x,y):=\big\{(y_1,y_2)\in Y\times Y~\big|~y_1\in S_1(x),~y_2\in S_2(x),~y_1+y_2=y\big\},
\end{equation}
fix $(\bar y_1,\bar y_2)\in G(\bar x,\bar y)$, and assume that $G$ is inner semicontinuous relative to $\Omega\times Y$ at $(\bar x,\bar y,\bar y_1,\bar y_2)$. Assume also that $\operatorname{gph}S_1|_\Omega$ and $\operatorname{gph}S_2|_\Omega$ are closed around
$(\bar x,\bar y_1)$ and $(\bar x,\bar y_2)$, respectively, that either $S_1$ is PSNC relative to $\Omega$ at $(\bar x,\bar y_1)$ or $S_2$ is PSNC relative to $\Omega$ at $(\bar x,\bar y_2)$, and that $\{S_1,S_2\}$ satisfies the qualification condition
\begin{equation}\label{sumQC}
D_M^*(S_1)_\Omega(\bar x\vbl\bar y_1)(0)\cap(-D_M^*(S_2)_\Omega(\bar x\vbl\bar y_2))(0)=\{0\}.
\end{equation}
Then for all $y^*\in Y$ we have the sum rule inclusions
\begin{eqnarray}
&&D_M^*(S_1+S_2)_\Omega(\bar x\vbl\bar y)(y^*)\subset D_M^*(S_1)_\Omega(\bar x\vbl\bar y_1)(y^*)+D_M^*(S_2)_\Omega(\bar x\vbl\bar y_2)(y^*),\label{sumrule1}\\
&&D_N^*(S_1+S_2)_\Omega(\bar x\vbl\bar y)(y^*)\subset D_N^*(S_1)_\Omega(\bar x\vbl\bar y_1)(y^*)+D_N^*(S_2)_\Omega(\bar x\vbl\bar y_2)(y^*).\label{sumrule2}
\end{eqnarray}
\end{theorem}\vspace*{-0.03in}

A striking consequence of Theorem~\ref{sumrule} and our major characterizations of well-posedness in Theorem~\ref{criterion} is the following assertion saying that the sum rules for the contingent coderivatives hold {\it unconditionally} in the general setting provided that at least one of the multifunctions is {\it relative Lipschitz-like} around the reference point.\vspace*{-0.03in}

\begin{corollary}\label{sumrileLip} In the general setting of Theorem~{\rm\ref{sumrule}}, with qualification condition \eqref{sumQC} imposed, suppose that at least one of the mapping $S_i$, $i=1,2$ is Lipschitz-like relative to $\Omega$ around the corresponding point. Then both coderivative sum rules in \eqref{sumrule1} and \eqref{sumrule2} are fulfilled.
\end{corollary}
\begin{proof}  Follows from Theorem~\ref{sumrule} and the equivalence in Theorem~\ref{criterion}(i,iii). 
\end{proof}

It is worth noting that with the new estimate of the constrained coderivative via the traditional mixed coderivative in Corollary \ref{upper-suff}, it is possible to obtain some upper estimates of chain rules and sum rules in Theorems 4.2 and 4.8, respectively.\vspace*{0.02in}

Finally in this section, we establish the {\it preservation} of the {\it relative PSNC property} under summation of multifunctions. The proof of the following proposition can still refer to the proof of \cite[Theorem 3.10]{Mordukhovich2006}. In fact, the proof is even simpler as $\lambda_k=1$ can be demonstrated directly without resorting to proof by contradiction. For this reason, we likewise omit the proof here.\vspace*{-0.05in}

\begin{proposition}\label{sum-psnc} In the setting of Theorem~{\rm\ref{sumrule}},  suppose that both multifunctions  $S_1$ and $S_2$ are PSNC relative to $\Omega$ at $(\bar x,\bar y_1)$ and $(\bar x,\bar y_2)$, respectively. Then their sum $S_1+S_2$ is PSNC relative to $\Omega$ at $(\bar x,\bar y_1 + \bar y_2)$.
\end{proposition}\vspace*{-0.1in}

\section{Concluding Remarks}\label{sec:conc} In this paper, we introduced and studied novel {\it relative contingent coderivatives} for set-valued mappings between Banach spaces. These coderivatives, being new even in finite dimensions, explore the ideas of combining {\it primal and dual} constructions that have never been tasted in infinite-dimensional variational analysis. In this way, we were able to establish {\it complete  characterizations} of {\it well-posedness} properties for general constrained systems and to develop comprehensive {\it calculus rules} for relative contingent coderivatives and relative PSNC properties.

Our future research will be addressing further applications of constrained coderi- vatives and variational calculus to challenging problems of optimization and optimal control. {  It is noted that the projectional coderivative introduced in \cite{Yang2021} was applied to derive a sufficient condition, an extension of the critical face condition in \cite{dontchev1996}, for solutions of affine variational inequalities to be relative Lipschitz-like in \cite{Yao2023} in finite-dimensional spaces. The mixed contingent coderivative introduced in this paper can be applied to derive similar sufficient conditions as in \cite{Yao2023} for polyhedral sets in reflexive Banach spaces. However, as pointed out by one referee, more suitable tools in infinite-dimensional spaces are the generalized polyhedral set in \cite[Definition 2.195]{Bonnans2000} and the polyhedric set, see \cite{Wachsmuth2019}. We will investigate the application of these polyhedral sets in the investigation of the relative Lipchitz-like property of solutions for affine variational inequalities in combination with the mixed contingent coderivative.} \vspace*{0.03in}

{\bf Acknowledgements}. The authors are indebted to the handling Associate Editor and two anonymous referees for their valuable and constructive remarks that allowed us to improve the original presentation.

\end{document}